\documentclass[12pt]{amsart}
\usepackage{amsmath}
\usepackage{amsthm}
\usepackage{amssymb}
\usepackage{latexsym}               
\usepackage{amsgen}
\usepackage[mathscr]{eucal}
\numberwithin{equation}{section}

\newcommand{\C}{\mathbb{C}}
\newcommand{\R}{\mathbb{R}}
\newcommand{\M}{\mathbb{M}}

\newcommand{\CE}{\mathcal{E}}

\usepackage{color}

\newcommand{\la}{\lambda}

\newcommand{\om}{\omega}

\newtheorem{lem}{Lemma}[section]
\newtheorem{thm}[lem]{Theorem}

\newtheorem{defn}[lem]{Definition}
\newtheorem{clm}[lem]{Claim}
\newtheorem{remark}{Remark}

\begin{document}

\title[Symmetric Hyperbolic Systems with Non-symmetric Relaxation]
{Decay structure for symmetric hyperbolic systems \\
with non-symmetric relaxation and its application
}

\author[Y. Ueda]{Yoshihiro Ueda}
\address{(YU)
Graduate School of Maritime Sciences, Kobe University, Kobe 658-0022, Japan}
\email{ueda@maritime.kobe-u.ac.jp}

\author[R.-J. Duan]{Renjun Duan}
\address{(RJD)
Department of Mathematics, The Chinese University of Hong Kong, Shatin, Hong Kong}
\email{rjduan@math.cuhk.edu.hk}

\author[S. Kawashima]{Shuichi Kawashima}
\address{(SK)
Faculty of Mathematics, Kyushu University, Fukuoka 819-0395, Japan}
\email{kawashim@math.kyushu-u.ac.jp}


\maketitle

\def\cal#1{{\fam2#1}}


%
\begin{abstract}
This paper is concerned with the decay structure for linear symmetric hyperbolic
systems with relaxation.
When the relaxation matrix is symmetric, the dissipative structure of the
systems is completely characterized by the Kawashima-Shizuta stability
condition formulated in \cite{UKS84,SK85}, and we obtain the asymptotic stability
result together with the explicit time-decay rate under that stability condition.
However, some physical models which
satisfy the stability condition have non-symmetric relaxation term
(cf.~the Timoshenko system and the Euler-Maxwell system).
Moreover, it had been already known that
the dissipative structure of such systems is weaker than the standard type
and is of the regularity-loss type
(cf.~\cite{D,IHK08,IK08,USK,UK}).
Therefore our purpose of this paper is to formulate a new structural condition
which include the Kawashima-Shizuta condition, 
and to analyze the weak dissipative structure for general systems with
non-symmetric relaxation.
\end{abstract}
\medskip

\qquad\qquad\qquad
Energy method

\medskip


\tableofcontents
\thispagestyle{empty}

\section{Introduction}

Consider the Cauchy problem for the first-order linear symmetric hyperbolic
system of equations with relaxation:
\begin{gather}
A^0 u_t + \sum_{j=1}^n A^j u_{x_j} + Lu = 0 \label{sys1}
\end{gather}
%
%
%
with
\begin{equation}\label{ID}
    u|_{t=0}=u_0.
\end{equation}
Here
$u=u(t,x) \in \R^m$ over  $t >0$, $x \in \mathbb{R}^n$ is an unknown function,
$u_0=u_0(x)\in \R^m$ over $x\in \R^n$ is a given function,
and $A^j$ $(j=0,1,\cdots,n)$ and $L$ are $m\times m$ real constant matrices,
where integers $m\geq 1,\,n\geq 1$ denote dimensions.
Throughout this paper, it is assumed that all $A^j$ $(j=0,1,\cdots,n)$ are symmetric,
$A^0$ is positive definite and $L$ is nonnegative definite with a nontrivial kernel.
Notice that $L$ is not necessarily symmetric.
For this general linear degenerately dissipative system it is interesting to study
its decay structure under additional conditions on the coefficient matrices
and further investigate the corresponding time-decay property of solutions to
the Cauchy problem.

When the degenerate relaxation matrix $L$ is symmetric,
Umeda-Kawashima-Shizuta \cite{UKS84} proved the large-time asymptotic stability
of solutions for a class of equations of hyperbolic-parabolic type with applications
to both electro-magneto-fluid dynamics and magnetohydrodynamics.
The key idea in \cite{UKS84} and the later generalized work \cite{SK85}
that first introduced the so-called Kawashima-Shizuta condition is
to design the compensating matrix to capture the dissipation
of systems over the degenerate kernel space of $L$. The typical
feature of the time-decay property of solutions established in those work
is that the high frequency part decays exponentially while the low frequency part
decays polynomially with the rate of the heat kernel.

For clearness and for later use let us precisely recall the results in
\cite{UKS84,SK85} mentioned above.
Taking the Fourier transform of \eqref{sys1} with respect to $x$ yields
\begin{gather}
A^0 \hat u_t + i |\xi| A(\omega) \hat u + L \hat u = 0. \label{Fsys1}
\end{gather}
Here and hereafter, $\xi\in{\mathbb R}^n$ denotes the Fourier variable,
$\omega=\xi/|\xi|\in S^{n-1}$ is the unit vector whenever $\xi\neq 0$, and
we define $A(\omega) := \sum_{j=1}^n A^j\omega_j$
with $\omega = (\omega_1, \cdots, \omega_n) \in S^{n-1}$.
The following two conditions for the coefficient matrices are needed:

\medskip

\noindent
{\bf Condition (A)$_0$:}\ \
$A^0$ is real symmetric and positive definite, $A^j$ $(j=1,\cdots,n)$ are
real symmetric, and $L$ is real symmetric and nonnegative definite with
the nontrivial kernel.

\medskip

\noindent
Namely, we assume that
%
\begin{equation*}
\begin{split}
&(A^j)^T = A^j \quad \text{for} \quad j= 0,1, \cdots , n,
\qquad L^T =L, \\[1mm]
&A^0 > 0, \quad L \ge 0 \quad \text{on} \quad \C^m,
\qquad {\rm Ker}(L) \neq \{0\}.
\end{split}
\end{equation*}
Here and in the sequel, the superscript $T$ stands for the transpose of matrices,
and given a matrix $X$, $X\geq 0$ means that
${\rm Re}\, \langle  X z, z \rangle \geq 0$ for any
$z\in{\mathbb C}^m$, while $X>0$ means that
${\rm Re}\, \langle  X z, z \rangle > 0$ for any
$z\in{\mathbb C}^m$ with $z\neq 0$, where $\langle \cdot, \cdot \rangle$
denotes the standard complex inner product in ${\mathbb C}^m$.
Also, for simplicity of notations, given a real matrix $X$, we use
$X_1$ and $X_2$ to denote the symmetric and skew-symmetric parts of $X$,
respectively, namely, $X_1=(X+X^T)/2$ and $X_2=(X-X^T)/2$.

\medskip

\noindent
{\bf Condition (K):}\ \
There is a real compensating matrix
$K(\omega) \in C^{\infty}(S^{n-1})$ with the following properties:
$K(-\omega) = - K(\omega)$, $(K(\omega)A^0)^T = - K(\omega)A^0$ and
\begin{equation}\label{assump-4*}
(K(\omega)A(\omega))_1 >0 \quad \text{on} \quad {\rm Ker}(L)
\end{equation}
for each $\omega\in S^{n-1}$.

\begin{remark}\label{remK1}
Under the condition {\rm (A)$_0$}, the positivity \eqref{assump-4*} in
the condition {\rm (K)} holds if and only if
\begin{equation}\label{assump-4**}
\alpha (K(\omega)A(\omega))_1 + L >0 \quad \text{on} \quad \C^m
\end{equation}
for each $\omega\in S^{n-1}$, where $\alpha$ is a suitably small positive
constant.
\end{remark}

This remark is easily verified as follows. First, we assume
\eqref{assump-4**} and suppose that $\phi\in {\rm Ker}(L)$.
Then, noting that $L\phi=0$, we have
\begin{equation*}
\alpha \langle (K(\omega)A(\omega))_1 \phi, \phi \rangle
= \langle (\alpha(K(\omega)A(\omega))_1 + L) \phi, \phi \rangle
\ge c|\phi|^2
\end{equation*}
for some positive constant $c$, where $\alpha$ is the positive constant in
\eqref{assump-4**}. This shows that \eqref{assump-4**} implies \eqref{assump-4*}.

Next, assuming \eqref{assump-4*}, we show \eqref{assump-4**}.
Let $\phi\in{\mathbb C}^m$ and let $P$ denote the orthogonal projection onto
${\rm Ker}(L)$. We have the decomposition $\phi=P\phi+(I-P)\phi$.
Then the positivity \eqref{assump-4*} on ${\rm Ker}(L)$ yields
$\langle(K(\omega)A(\omega))_1\phi,\phi\rangle
\geq c|P\phi|^2-C|(I-P)\phi|^2$, where $c$ and $C$ are some positive
constants.
Also, we have $\langle L\phi,\phi\rangle\geq c|(I-P)\phi|^2$ for a positive
constant $c$.
Now, letting $\alpha>0$, we can compute as
\begin{equation*}
\langle (\alpha(K(\omega)A(\omega))_1 + L) \phi, \phi \rangle
\geq \alpha c|P\phi|^2+(c-\alpha C)|(I-P)\phi|^2
\geq c_1|\phi|^2,
\end{equation*}
where we choose $\alpha>0$ so small that $\alpha C\leq c/2$, and $c_1$
is a positive constant satisfying $c_1\leq{\rm min}\{\alpha c, c/2\}$.
Thus we have shown that \eqref{assump-4*} implies \eqref{assump-4**}.
This completes the proof of Remark \ref{remK1}.

%
%
%


\medskip

Under the conditions (A)$_0$ and (K) one has:

\begin{thm}
[Decay property of the standard type (\cite{UKS84,SK85})]\label{thm}
%
Assume that both the conditions {\rm (A)$_0$} and {\rm (K)} hold.
Then the Fourier image $\hat u$ of the solution $u$ to the Cauchy problem
\eqref{sys1}-\eqref{ID} satisfies the pointwise estimate:
\begin{equation}\label{std-point}
|\hat u(t,\xi)| \le Ce^{-c \rho(\xi)t}|\hat u_0(\xi)|,
\end{equation}
where $\rho(\xi) := |\xi|^2/(1+|\xi|^2)$.
Furthermore, let $s \ge 0$ be an integer and suppose that the initial data
$u_0$ belong to $H^s \cap L^1$.
Then the solution $u$ satisfies the decay estimate: 
\begin{equation}\label{std-decay}
\|\partial_x^{k} u(t)\|_{L^2} \le C(1+t)^{-n/4-k/2}\|u_0\|_{L^1}
	+ Ce^{-ct}\|\partial_x^{k} u_0\|_{L^2}
\end{equation}
for $k\le s$. Here $C$ and $c$ are positive constants.
\end{thm}

Unfortunately, when the degenerate relaxation matrix $L$ is not symmetric,
Theorem \ref{thm} can not be applied any longer.
In fact, this is the case for some concrete systems, for example,
the Timoshenko system \cite{IHK08,IK08} and the Euler-Maxwell system
\cite{D,USK,UK}, where the linearized relaxation matrix $L$ indeed has
a nonzero skew-symmetric part while
it was still proved that solutions decay in time in some different way that
we shall point out later on.
Therefore, our purpose of this paper is to formulate some new structural
conditions in order to extend Theorem \ref{thm} to the general system
\eqref{sys1} when $L$ is not symmetric, which can include both
the Timoshenko system and the Euler-Maxwell system.

More precisely, we introduce a constant matrix $S$ which satisfies some
properties in Condition (S) in Section 2.
When the relaxation matrix $L$ is not symmetric, we have a partial positivity
on ${\rm Ker}(L_1)^\perp$ only. In this situation, we try finding a real
compensating matrix $S$ to make a positivity on ${\rm Ker}(L)^\perp$.
Then, employing further the condition (K), we can construct a full positivity
on $\C^m$.
As the consequence, we can show the following weaker estimates:
\begin{equation}\label{1.8}
|\hat u(t,\xi)| \le Ce^{-c \eta(\xi)t}|\hat u_0(\xi)|,
\end{equation}
where $\eta(\xi) := |\xi|^2/(1+|\xi|^2)^2$, and
\begin{equation}\label{1.9}
\|\partial_x^{k} u(t)\|_{L^2} \le C(1+t)^{-n/4-k/2}\|u_0\|_{L^1}
	+ C(1+t)^{-\ell/2}\|\partial_x^{k+\ell} u_0\|_{L^2}
\end{equation}
for $k+\ell \le s$. See Theorem \ref{thm1} for the details.
We note that these estimates \eqref{1.8} and \eqref{1.9} are weaker than
\eqref{std-point} and \eqref{std-decay}, respectively.
In particular, the decay estimate \eqref{1.8} is of the regularity-loss type.

Similar decay properties of the regularity-loss type have been recently
observed for several interesting systems. We refer the reader to
\cite{IHK08,IK08,LK4} (cf. \cite{ABMR,MRR}) for the dissipative Timoshenko system,
\cite{D,USK,UK} for the Euler-Maxwell system,
\cite{HK,KK} for a hyperbolic-elliptic system in radiation gas dynamics,
\cite{LK1,LK2,LK3,LC,SK} for a dissipative plate equation, and
\cite{Du-1VMB,DS-VMB} for various kinetic-fluid models.

\smallskip

The contents of this paper are as follows.
In Section 2 we formulate several structural conditions and state our
main results on the decay property of the system \eqref{sys1} when the
relaxation matrix $L$ is not symmetric.
The obtained decay estimates are of the regularity-loss type.
In Section 3 we develop the energy method in the Fourier space and derive
the pointwise estimates for the Fourier image of the solution, which is
crucial in showing our decay estimates.
In Section 4 we discuss the relationship between the structural conditions.
In particular, we show that the rank condition (R) in \cite{BZ11} is a
sufficient condition for the condition (K) even if $L$ is not symmetric.
The decay property of the system \eqref{sys1} with constraints is
investigated in Section 5.
Finally, in Sections 6 and 7, we treat the Timoshenko system and the
Euler-Maxwell system as applications of our general theory.

\smallskip

\noindent
{\bf Notations.}\ \
For a nonnegative integer $k$, we denote by $\partial^k_x$
the totality of all the $k$-th order derivatives with respect
to $x = (x_1, \cdots ,x_n)$.

Let $1\leq p\leq\infty$. Then $L^p=L^p({\mathbb R}^n)$ denotes
the usual Lebesgue space over ${\mathbb R}^n$ with the norm
$\|\cdot\|_{L^p}$. For a nonnegative integer $s$,
$H^{s}=H^{s}({\mathbb R}^n)$ denotes
the $s$-th order Sobolev space over ${\mathbb R}^n$ in the $L^2$
sense, equipped with the norm $\|\cdot\|_{H^{s}}$.
We note that 
$L^2=H^0$.

Finally, in this paper, we use $C$ or $c$ to denote various positive
constants without confusion.


\section{Decay structure}

In this section we shall introduce new structural conditions to investigate
the decay structure and time-decay property for the system \eqref{sys1}
when $L$ is not necessarily symmetric,
and then state under those conditions the main results which are the
generalization of Theorem \ref{thm}.
Our structural conditions are formulated as follows.

\medskip

\noindent
{\bf Condition (A):}\ \
$A^0$ is real symmetric and positive definite, $A^j$ $(j=1,\cdots,n)$ are
real symmetric, while $L$ is not necessarily real symmetric but is nonnegative
definite with the nontrivial kernel.

\medskip

\noindent
Namely, it is assumed that
\begin{equation*}
\begin{split}
&(A^j)^T = A^j \quad \text{for} \quad j= 0,1, \cdots , n,  \\[1mm]
&A^0 > 0, \quad L \ge 0 \quad \text{on} \quad \C^m,
\qquad {\rm Ker}(L) \neq \{0\}.
\end{split}
\end{equation*}
%




\medskip

\noindent
{\bf Condition (S):}\ \
There is a real constant matrix $S$ with the following properties:
$(SA^0)^T = SA^0$ and
%
\begin{gather}
(SL)_1+L_1\geq 0 \quad {\rm on} \quad {\mathbb C}^m, \qquad
{\rm Ker}((SL)_1+L_1) = 
{\rm Ker}(L). \label{assump-2}
\end{gather}

\begin{remark}\label{remK2}
Under the conditions {\rm (A)} and {\rm (S)}, the positivity \eqref{assump-4*}
in the condition {\rm (K)} holds if and only if
\begin{equation}\label{assump-4}
\alpha (K(\omega)A(\omega))_1+(SL)_1+L_1>0 \quad \text{on} \quad \C^m
\end{equation}
for each $\omega\in S^{n-1}$, where $\alpha$ is a suitably small positive
constant.
\end{remark}

In fact, by virtue of \eqref{assump-2}, we find that
\begin{equation*}
\langle ((SL)_1 + L_1) \phi, \phi \rangle
\ge c |(I-P) \phi|^2
\end{equation*}
for any $\phi \in \C^m$, where $c$ is a positive constant and $P$ denotes
the orthogonal projection onto ${\rm Ker}(L)$.
Using this property, we can show the equivalence of \eqref{assump-4*} and
\eqref{assump-4} in a similar way as in the proof of Remark \ref{remK1}.

When we use the condition (S), we additionally assume either the condition
(S)$_1$ or (S)$_2$ below.

\medskip

\noindent
{\bf Condition (S)$_1$:}\ \
For each $\omega\in S^{n-1}$, the matrix $S$ in the condition (S) satisfies
%
\begin{gather}
i(SA(\omega))_2\geq 0 \quad \text{on} \quad {\rm Ker}(L_1). \label{assump-3}
\end{gather}

%

\medskip

\noindent
{\bf Condition (S)$_2$:}\ \
For each $\omega\in S^{n-1}$, the matrix $S$ in the condition (S) satisfies
\begin{equation}\label{assump-5}
i(SA(\omega))_2 \ge 0 \quad \text{on} \quad \C^m.
\end{equation}

\medskip

Under the above structural conditions, we can state our main results on
the decay property for the system \eqref{sys1}. The first one uses the
condition (S)$_1$.

\begin{thm}
[Decay property of the regularity-loss type]\label{thm1}
Assume that the conditions {\rm (A)}, {\rm (S)}, {\rm (S)$_1$} and
{\rm (K)} hold.
Then the Fourier image $\hat u$ of the solution $u$ to the Cauchy problem
\eqref{sys1}-\eqref{ID} satisfies the pointwise estimate:
\begin{equation}\label{loss-point}
|\hat u(t,\xi)| \le Ce^{-c \eta(\xi)t}|\hat u_0(\xi)|,
\end{equation}
where $\eta(\xi) := |\xi|^2/(1+|\xi|^2)^2$.
Moreover, let $s \ge 0$ be an integer and suppose that the initial data
$u_0$ belong to $H^s \cap L^1$.
Then the solution $u$ satisfies the decay estimate: 
\begin{equation}\label{loss-decay}
\|\partial_x^{k} u(t)\|_{L^2} \le C(1+t)^{-n/4-k/2}\|u_0\|_{L^1}
	+ C(1+t)^{-\ell/2}\|\partial_x^{k+\ell} u_0\|_{L^2}
\end{equation}
for $k+\ell \le s$. Here $C$ and $c$ are positive constants.
\end{thm}

\begin{remark}
The decay estimate \eqref{loss-decay} is of the regularity-loss type
because we have the decay rate $(1+t)^{-\ell/2}$ only by assuming the
additional $l$-th order regularity on the initial data.
\end{remark}

Our second main result uses the stronger condition (S)$_2$ instead of
(S)$_1$ and gives the decay estimate of the standard type.

\begin{thm}
[Decay property of the standard type]\label{thm2}
If the condition {\rm (S)$_1$} in Theorem \ref{thm1} is replaced by the
stronger condition {\rm (S)$_2$}, then the pointwise estimate
\eqref{loss-point} and the decay estimate \eqref{loss-decay} in
Theorem \ref{thm1} can be refined as \eqref{std-point} and \eqref{std-decay}
in Theorem \ref{thm}, respectively.
%
%
%
%
\end{thm}

It should be pointed out that Theorem \ref{thm2} is a direct extension of
Theorem \ref{thm} and is applicable to the system \eqref{sys1} with a
non-symmetric relaxation matrix $L$. More specifically, we have:

\begin{clm}
Theorem \ref{thm} holds as a corollary of Theorem \ref{thm2}.
In other words, when $L$ is real symmetric, Theorem \ref{thm2} is
reduced to Theorem \ref{thm}.
\end{clm}

In fact, when $L$ is real symmetric, the condition (A) is reduced to
(A)$_0$. Moreover, in this case, we have $L=L_1$ so that the conditions
(S) and (S)$_2$ are satisfied trivially with $S=0$.
This shows that Theorem \ref{thm2} implies Theorem \ref{thm}.

\medskip

Next we introduce the rank condition (R) which was formulated by Beauchard
and Zuazua in \cite{BZ11}.

\medskip

\noindent
{\bf Condition (R):}\ \
The matrices $A^0$, $A(\om)$ and $L$ satisfies the following rank condition:
\begin{equation}\label{thm.ra.c4}
    {\rm Rank}\,\left[
    \begin{array}{c}
      L\\
      L \tilde A(\om) \\
      \vdots\\
      L \tilde A(\om)^{m-1}\\
    \end{array}
    \right]=m
\end{equation}
for each $\omega\in S^{n-1}$, where $\tilde A(\om) := (A^0)^{-1}A(\om)$.

\medskip

\noindent
This condition (R) is called the Kalman rank condition in the control theory
and is proved to be equivalent to the condition (K) under the condition
(A)$_0$ where $L$ is real symmetric. For the details, see \cite{BZ11}.
In our case where $L$ is not necessarily real symmetric, under the condition
(A), we can show that the condition (R) implies the condition (K);
see Theorem \ref{thmK} in Section 4. Consequently, we have the following claim.

\begin{clm}
In Theorems \ref{thm1} and \ref{thm2} above, we can replace the condition
{\rm (K)} by the rank condition {\rm (R)}.
\end{clm}

In Theorems \ref{thm1} and \ref{thm2},
the decay estimates \eqref{loss-decay} and \eqref{std-decay}
can be derived by using the pointwise estimates \eqref{loss-point} and
\eqref{std-point}, respectively.
Before closing this section, we prove this fact.

\medskip

\noindent
{\bf Proof of the decay estimates in Theorems \ref{thm1} and \ref{thm2}.}\ \
We first prove \eqref{loss-decay} in Theorem \ref{thm1}.
By virtue of the Plancherel theorem and the pointwise estimate
\eqref{loss-point}, we obtain
\begin{equation}\label{pf1}
\|\partial_x^k u(t)\|^2_{L^2}
= \int_{\mathbb{R}^n}|\xi|^{2k}|\hat u(t,\xi)|^2 d\xi
\le C \int_{\mathbb{R}^n}|\xi|^{2k}
   e^{-c \eta(\xi)t}|\hat u_0(\xi)|^2 d\xi.
\end{equation}
We divide the integral on the right-hand side of \eqref{pf1}
into two parts $I_1$ and $I_2$
according to the low frequency region $|\xi| \le 1$ and
the high frequency region $|\xi| \ge 1$, respectively.
Since $\eta(\xi) \ge c|\xi|^2$ for $|\xi| \le 1$, we see that
\begin{equation*}
I_1
\le C \sup_{|\xi|\le1} |\hat u_0(\xi)|^2
   \int_{|\xi|\le 1}|\xi|^{2k}e^{-c |\xi|^2t} d\xi
\le C(1+t)^{-n/2-k}\|u_0\|^2_{L^1}.
\end{equation*}
On the other hand, we have $\eta(\xi) \ge c|\xi|^{-2}$ in the
region $|\xi| \ge 1$. Consequently, we obtain
\begin{equation*}
I_2
\le C \sup_{|\xi| \ge 1} \frac{e^{-c t/|\xi|^2}}{|\xi|^{2\ell}}
   \int_{|\xi|\ge 1}|\xi|^{2(k+\ell)}|\hat u_0(\xi)|^2 d\xi
\le C(1+t)^{-\ell}\|\partial_x^{k + \ell}u_0\|^2_{L^2}.
\end{equation*}
Therefore, substituting these estimates into \eqref{pf1},
we get the desired decay estimate \eqref{loss-decay}.

To prove \eqref{std-decay} in Theorem \ref{thm2}, we make use of
the pointwise estimate \eqref{std-point}.
Since $\rho(\xi) \ge c|\xi|^2$ for $|\xi| \le 1$ and
$\rho(\xi) \ge c$ for $|\xi| \ge 1$, a similar computation
as in the proof of \eqref{loss-decay} yields the decay estimate
\eqref{std-decay}.
Thus we got the desired decay estimates
and this completes the proof.
\qed


\section{Energy method in the Fourier space}

The aim of this section is to prove the pointwise estimates stated in
Theorems \ref{thm1} and \ref{thm2} by employing the energy method in the
Fourier space.

\medskip

\noindent
{\bf Proof of the pointwise estimate in Theorem \ref{thm1}.}\ \ 
We derive the energy estimate for the system \eqref{Fsys1} in the Fourier space.
Taking the inner product of \eqref{Fsys1} with $\hat{u}$, we have
\begin{equation*}
	\langle A^0 \hat u_t, \hat{u} \rangle
	+ i |\xi| \langle A(\omega)\hat u, \hat u \rangle
	+ \langle L \hat u, \hat u \rangle = 0.
\end{equation*}
Taking the real part, we get the basic energy equality
\begin{equation}\label{eq}
\frac{1}{2}\frac{d}{dt} E_0
+ \langle L_1 \hat u, \hat u \rangle = 0,
\end{equation}
where
$
E_0 := \langle A^0 \hat u, \hat u \rangle.
$
Next we create the dissipation terms.
For this purpose, we multiply \eqref{Fsys1} by the matrix $S$ in the
condition (S) and take the inner product with $\hat u$.
This yields
\begin{equation*}
	\langle SA^0 \hat u_t, \hat u \rangle
	+ i |\xi| \langle SA(\omega)\hat u, \hat u \rangle
	+ \langle SL \hat u, \hat u \rangle = 0.
\end{equation*}
Taking the real part of this equality, we get
\begin{equation}\label{eqS}
\frac{1}{2}\frac{d}{dt} E_1
+ |\xi| \langle i(SA(\omega))_2\hat u, \hat u \rangle
+ \langle (S L)_1 \hat u, \hat u \rangle = 0,
\end{equation}
where
$
E_1 := \langle SA^0 \hat u, \hat u \rangle.
$
Moreover, letting $K(\omega)$ be the compensating matrix in the condition
(K), we multiply \eqref{Fsys1} by $- i|\xi|K(\omega)$ and take
the inner product with $\hat u$.
Then we have
\begin{equation*}
-i|\xi|\langle K(\omega)A^0 \hat u_t, \hat u \rangle
+ |\xi|^2 \langle K(\omega)A(\omega)\hat u, \hat u \rangle
- i|\xi|\langle K(\omega)L \hat u, \hat u \rangle = 0.
\end{equation*}
Taking the real part of the above equality, we obtain
\begin{equation}\label{eqK}
- 
\frac{1}{2}|\xi| \frac{d}{dt}E_2
+ |\xi|^2 \langle (K(\omega)A(\omega))_1\hat u, \hat u \rangle
- |\xi|\langle i(K(\omega)L)_2 \hat u, \hat u \rangle = 0,
\end{equation}
where
$
E_2:= \langle i K(\omega)A^0 \hat u, \hat u \rangle.
$

Now we combine the energy equalities \eqref{eq}, \eqref{eqS} and \eqref{eqK}.
First, letting $\alpha$ be the positive number in Remark \ref{remK2},
we multiply \eqref{eqS} and \eqref{eqK} by $1+|\xi|^2$ and
$\alpha_2\alpha$, respectively, and add these two equalities, where
$\alpha_2$ is a positive constant to be determined. This yields
\begin{equation}\label{eqSK}
\begin{split}
&\frac{1}{2} 
   (1+|\xi|^2) \frac{d}{dt} \CE
+ (1+|\xi|^2)\langle (S L)_1 \hat u, \hat u \rangle
   + \alpha_2|\xi|^2\langle \alpha(K(\omega)A(\omega))_1\hat u, \hat u \rangle
\\[2mm]
&= - |\xi|(1+|\xi|^2) \langle i(SA(\omega))_2 \hat u, \hat u \rangle
   + \alpha_2|\xi| \langle i\alpha(K(\omega)L)_2 \hat u, \hat u \rangle,
\end{split}
\end{equation}
where
$
\CE := E_1 - \frac{\alpha_2|\xi|}{1+|\xi|^2}\alpha E_2.
$
Furthermore, we multiply \eqref{eq} and \eqref{eqSK} by $(1+|\xi|^2)^2$
and $\alpha_1$, respectively, and add the resulting two equalities,
where $\alpha_1$ is a positive constant to be determined. This yields
\begin{equation}\label{eq-final}
\begin{split}
&\frac{1}{2} (1+|\xi|^2)^2 \frac{d}{dt}
   \Big( E_0 + \frac{\alpha_1}{1+|\xi|^2}\,\CE \Big) \\[1mm]
&+(1+|\xi|^2)^2 \langle L_1 \hat u, \hat u \rangle
   + \alpha_1\big\{
   (1+|\xi|^2)\langle (S L)_1 \hat u, \hat u \rangle
   + \alpha_2|\xi|^2\langle \alpha(K(\omega)A(\omega))_1\hat u, \hat u \rangle
   \big\} \\[1mm]
&= \alpha_1\big\{
   - |\xi|(1+|\xi|^2) \langle i(SA(\omega))_2 \hat u, \hat u \rangle
   + \alpha_2 |\xi| \langle i\alpha(K(\omega)L)_2 \hat u, \hat u \rangle
   \big\}.
\end{split}
\end{equation}
We write the equality \eqref{eq-final} as
\begin{equation}\label{energy}
\frac{1}{2}\frac{d}{dt}E + D_1 + D_2 = G,
\end{equation}
where we define $E$, $D_1$, $D_2$ and $G$ as
\begin{equation}\label{note-energy}
\begin{split}
&E := E_0 + \frac{\alpha_1}{1+|\xi|^2}\,\CE
  =E_0 + \frac{\alpha_1}{1+|\xi|^2}\Big(
  E_1+\frac{\alpha_2|\xi|}{1+|\xi|^2}\,\alpha E_2\Big), \\[1mm]
&(1+|\xi|^2)^2D_1
   := (1+|\xi|^2)^2 \langle L_1 \hat u, \hat u \rangle \\
&\qquad\qquad\qquad
   + \alpha_1\big\{(1+|\xi|^2)\langle (S L)_1 \hat u, \hat u \rangle
   +\alpha_2|\xi|^2\langle \alpha(K(\omega)A(\omega))_1\hat u, \hat u \rangle \big\}, \\[1mm]
&(1+|\xi|^2)^2D_2
   := \alpha_1 |\xi|(1+|\xi|^2) \langle i(SA(\omega))_2 P_1 \hat u, P_1 \hat u \rangle, \\[1mm]
&(1+|\xi|^2)^2 G
   := \alpha_1 \alpha_2|\xi| \langle i\alpha(K(\omega)L)_2 \hat u, \hat u \rangle \\
&\qquad\qquad\qquad
   - \alpha_1 |\xi|(1+|\xi|^2) \big\{  \langle i(SA(\omega))_2 \hat u, \hat u \rangle
   - \langle i(SA(\omega))_2 P_1 \hat u, P_1 \hat u \rangle \big\}.
\end{split}
\end{equation}

We estimate each term in \eqref{energy}.
Because of the positivity of $A^0$,
for suitably small $\alpha_1>0$ and $\alpha_2>0$, we see that
\begin{equation}\label{est-1}
c_0|\hat u|^2\leq E \leq C_0|\hat u|^2,
\end{equation}
where $c_0$ and $C_0$ are positive constants not depending on
$(\alpha_1,\alpha_2)$. On the other hand, we can rewrite $D_1$ as
\begin{equation}\label{dissipation}
\begin{split}
(1+|\xi|^2)^2 D_1
&=\alpha_1\alpha_2|\xi|^2
   \langle (\alpha (K(\omega)A(\omega))_1+(SL)_1+L_1)\hat u, \hat u \rangle
\\[1mm]
&+\alpha_1((1+|\xi|^2)-\alpha_2|\xi|^2)
   \langle ((SL)_1+L_1)\hat u, \hat u \rangle \\[1mm]
&+(1+|\xi|^2)((1+|\xi|^2)-\alpha_1)
   \langle L_1 \hat u, \hat u \rangle.
\end{split}
\end{equation}
Here, using the positivity \eqref{assump-4} which is based on the condition
(K), we have
\begin{equation}\label{est-D}
\langle (\alpha(K(\omega)A(\omega))_1+(SL)_1+L_1)\hat u, \hat u \rangle
\ge c_1 |\hat u|^2,
\end{equation}
where $c_1$ is a positive constant. Therefore we can estimate $D_1$ as
\begin{equation}\label{est-2}
(1+|\xi|^2)^2D_1 \geq \alpha_1\alpha_2c_1|\xi|^2|\hat u|^2
   +\alpha_1c_2(1+|\xi|^2)|(I-P)\hat u|^2
   +c_3(1+|\xi|^2)^2|(I-P_1)\hat u|^2,
\end{equation}
where $c_1$ is the constant in \eqref{est-D}, $c_2$ and $c_3$ are positive
constants not depending on $(\alpha_1,\alpha_2)$, and $P$ and $P_1$ denote
the orthogonal projections onto ${\rm Ker}(L)$ and ${\rm Ker}(L_1)$,
respectively. Here we have used \eqref{assump-2} in the condition (S) and
the fact that $L_1\geq 0$ on ${\mathbb C}^m$ which is due to the condition
(A). Also we see that $D_2\geq 0$ by the condition (S)$_1$.

Finally, we estimate each term in $G$. Note that
\begin{equation*}
\langle i(K(\omega)L)_2 \hat u, \hat u \rangle
={\rm Re}\langle iK(\omega)L \hat u, \hat u \rangle
={\rm Re}\langle iK(\omega)L(I-P) \hat u, \hat u \rangle,
\end{equation*}
where we used $LP=0$. Thus we have
\begin{equation}\label{est-3a}
|\xi||\langle i\alpha(K(\omega)L)_2 \hat u, \hat u \rangle|
\leq C|\xi||(I-P)\hat u||\hat u| 
\leq \epsilon|\xi|^2|\hat u|^2+C_\epsilon|(I-P)\hat u|^2
\end{equation}
for any $\epsilon>0$, where $C_\epsilon$ is a constant depending on
$\epsilon$. For the remaining term in $G$, by using the equality
\begin{equation*}
\begin{split}
&\langle i(SA(\omega))_2 \hat u, \hat u \rangle -
\langle i(SA(\omega))_2 P_1 \hat u, P_1\hat u \rangle \\
&= \langle i(SA(\omega))_2 P_1 \hat u, (I-P_1)\hat u \rangle
   + \langle i(SA(\omega))_2 (I-P_1) \hat u, \hat u \rangle,
\end{split}
\end{equation*}
we estimate as
%
\begin{equation}\label{est-3c}
\begin{split}
|\xi|&(1+|\xi|^2)
   \big|\langle i(SA(\omega))_2 \hat u, \hat u \rangle -
\langle i(SA(\omega))_2 P_1 \hat u, P_1\hat u \rangle\big| \\[1mm]
&\leq C|\xi|(1+|\xi|^2) |(I-P_1)\hat u||\hat u| \\[1mm]
&\leq \delta|\xi|^2|\hat u|^2
   +C_\delta(1+|\xi|^2)^2|(I-P_1)\hat u|^2
\end{split}
\end{equation}
for any $\delta>0$, where $C_\delta$ is a constant depending on $\delta$.
Consequently, we obtain
\begin{equation}\label{est-G}
\begin{split}
(1+|\xi|^2)^2|G|
&\leq \alpha_1(\alpha_2\epsilon+\delta)|\xi|^2|\hat u|^2 \\[1mm]
&+\alpha_1\alpha_2C_\epsilon|(I-P)\hat u|^2
+\alpha_1C_\delta(1+|\xi|^2)^2|(I-P_1)\hat u|^2.
\end{split}
\end{equation}

We choose $\epsilon>0$ and $\delta>0$ such that $\epsilon=c_1/4$ and
$\delta=\alpha_2c_1/4$. For this choice of $(\epsilon,\delta)$,
we take $\alpha_2>0$ and $\alpha_1>0$ so small that
$\alpha_2C_\epsilon\leq c_2/2$ and $\alpha_1C_\delta\leq c_3/2$.
Then, by using \eqref{est-2}, \eqref{est-3a} and \eqref{est-3c},
we conclude that
$
|G|
\leq D_1/2
$
and
\begin{equation}\label{est-2a}
D_1\geq c\Big\{\frac{|\xi|^2}{(1+|\xi|^2)^2}|\hat u|^2
   + \frac{1}{1+|\xi|^2}|(I-P)\hat u|^2 + |(I-P_1)\hat u|^2\Big\},
\end{equation}
where $c$ is a positive constant.
Consequently, \eqref{energy} becomes
\begin{equation}\label{energy2}
\frac{d}{dt}E + D_1 + 2D_2 
\leq 0.
\end{equation}
%
%
Moreover, it follows from \eqref{est-1} and \eqref{est-2a} that
$D_1 \geq c\eta(\xi) E$,
where $\eta(\xi)=|\xi|^2/(1+|\xi|^2)^2$, and $c$ is a positive constant.
Also we have $D_2\geq 0$. Thus \eqref{energy2} leads the estimate
\begin{equation}\label{est-4}
\frac{d}{dt}E  +  c\eta(\xi) E \le 0.
\end{equation}
Solving this differential inequality, we get
$E(t,\xi)\leq e^{-c\eta(\xi)t}E(0,\xi)$, which together with
\eqref{est-1} gives the desired pointwise estimate \eqref{loss-point}.
This completes the proof of Theorem \ref{thm1}.
\qed

\medskip
When the condition (S)$_1$ is replaced by (S)$_2$, the above computations
can be simplified  and we obtain the better pointwise estimate
\eqref{std-point}.

\medskip
\noindent
{\bf Proof of the pointwise estimate in Theorem \ref{thm2}.}\ \ 
Under the assumption \eqref{assump-5} in the condition (S)$_2$,
the first term on the right-hand side of \eqref{eqSK} becomes a good term
and we obtain
\begin{equation}\label{eqSK-2}
\begin{split}
&\frac{1}{2} (1+|\xi|^2) \frac{d}{dt} \CE
+ (1+|\xi|^2)\langle (S L)_1 \hat u, \hat u \rangle
+  \alpha_2|\xi|^2\langle \alpha(K(\omega)A(\omega))_1\hat u, \hat u \rangle\\[1mm]
&\qquad
+ |\xi|(1+|\xi|^2) \langle i(SA(\omega))_2 \hat u, \hat u \rangle
 = \alpha_2|\xi| \langle i\alpha(K(\omega)L)_2 \hat u, \hat u \rangle.
\end{split}
\end{equation}
%
In this case, we multiply \eqref{eq} and \eqref{eqSK-2} by $1+|\xi|^2$
and $\alpha_1$, respectively, and combine the resultant two equalities.
This yields
%
%
%
\begin{equation}\label{energy-2}
\frac{1}{2} \frac{d}{dt} \tilde E + \tilde D_1 + \tilde D_2 = \tilde G,
\end{equation}
where we define as
\begin{equation}\label{ene}
\begin{split}
&\tilde E := E_0 + \alpha_1 \CE
  =E_0+\alpha_1\Big(E_1+\frac{\alpha_2|\xi|}{1+|\xi|^2}\,\alpha E_2\Big), \\[1mm]
&(1+|\xi|^2)\tilde D_1
  := (1+|\xi|^2)\langle L_1 \hat u, \hat u \rangle \\
&\qquad\qquad
   + \alpha_1\big\{(1+|\xi|^2)\langle (S L)_1 \hat u, \hat u \rangle
   +\alpha_2|\xi|^2 \langle \alpha(K(\omega)A(\omega))_1\hat u, \hat u \rangle
   \big\}, \\[1mm]
&\tilde D_2 := \alpha_1 |\xi| \langle i(SA(\omega))_2 \hat u, \hat u \rangle, \qquad 
   (1+|\xi|^2)\tilde G
   :=\alpha_1 \alpha_2|\xi|\langle i\alpha(K(\omega)L)_2 \hat u, \hat u \rangle.
\end{split}
\end{equation}
%
Here, for suitably small $\alpha_1>0$ and $\alpha_2>0$, we see that
\begin{equation}\label{est-1-2}
c_0|\hat u|^2\leq \tilde E\leq C_0|\hat u|^2,
\end{equation}
where $c_0$ and $C_0$ are positive constants not depending on
$(\alpha_1,\alpha_2)$. On the other hand, we can rewrite $\tilde D_1$ as
\begin{equation*}
\begin{split}
(1+|\xi|^2) \tilde D_1
&=\alpha_1\alpha_2|\xi|^2
   \langle (\alpha(K(\omega)A(\omega))_1+(SL)_1+L_1)\hat u, \hat u \rangle
\\[1mm]
&+\alpha_1((1+|\xi|^2)-\alpha_2|\xi|^2)
   \langle ((SL)_1+L_1)\hat u, \hat u \rangle
+(1-\alpha_1)(1+|\xi|^2)\langle L_1 \hat u, \hat u \rangle.
\end{split}
\end{equation*}
Then, 
as in the derivation of \eqref{est-2},
for suitably small $\alpha_1>0$ and $\alpha_2>0$,
we can estimate $\tilde D_1$ as
\begin{equation*}
(1+|\xi|^2)\tilde D_1 \geq \alpha_1\alpha_2c_1|\xi|^2|\hat u|^2
   +\alpha_1c_2(1+|\xi|^2)|(I-P)\hat u|^2
   +c_3(1+|\xi|^2)|(I-P_1)\hat u|^2,
\end{equation*}
where $c_1$, $c_2$ and $c_3$ are positive constants not depending
on $(\alpha_1,\alpha_2)$.
Also, making use of \eqref{est-3a}, we can estimate the term $\tilde G$ as
\begin{equation}\label{est-3aG}
(1+|\xi|^2)|\tilde G|
\leq \alpha_1\alpha_2\epsilon|\xi|^2|\hat u|^2
+\alpha_1\alpha_2C_\epsilon|(I-P)\hat u|^2
\end{equation}
for any $\epsilon>0$, where $C_\epsilon$ is a constant depending on
$\epsilon$ but not on $(\epsilon,\delta)$.

We choose $\epsilon > 0$ in \eqref{est-3aG} so small that
$\epsilon=c_1/2$. For this choice of $\epsilon$,
we take $\alpha_2>0$ so small that $\alpha_2C_\epsilon\leq c_2/2$.
Then we obtain $|\tilde G|\leq \tilde D_1/2$ and
\begin{equation}\label{est-2-2a}
\tilde D_1\geq c\Big\{\frac{|\xi|^2}{1+|\xi|^2}|\hat u|^2
+|(I-P)\hat u|^2+|(I-P_1)\hat u|^2\Big\},
\end{equation}
where $c$ is a positive constant. Consequently, \eqref{energy-2} becomes
\begin{equation*}
\frac{d}{dt}\tilde E+ \tilde D_1 + 2 \tilde D_2 \leq 0.
\end{equation*}
Here we note that $D_2\geq 0$ by \eqref{assump-5} in the condition (S)$_2$.
%
%
Also we have from \eqref{est-1-2} and \eqref{est-2-2a} that
$\tilde D_1 \geq\rho(\xi)\tilde E$, where $\rho(\xi)=|\xi|^2/(1+|\xi|^2)$,
and $c$ is a positive constant. Thus we obtain
%
$
\frac{d}{dt}\tilde E + c\rho(\xi)\tilde E\le 0,
$
%
which is solved as $\tilde E(t,\xi)\leq e^{-c\rho(\xi)t}\tilde E(0,\xi)$.
This together with \eqref{est-1-2} gives the desired pointwise estimate
\eqref{std-point}.
Thus the proof of Theorem \ref{thm2} is complete.
\qed


\section{Relation between structural conditions}

In this section we discuss the dissipative structure for the system
\eqref{sys1}.
To this end, we introduce a notion of the {\it uniform dissipativity} of
the system \eqref{sys1}. We consider the eigenvalue problem for the system
\eqref{sys1} or \eqref{Fsys1}:
\begin{equation}\label{eg}
(\lambda A^0 + i|\xi|A(\omega) + L) \phi = 0,
\end{equation}
where $\lambda \in \C$ and $\phi \in \C^m$.
The corresponding characteristic equation is given by
\begin{equation}\label{ce}
{\rm det}(\lambda A^0 + i|\xi|A(\omega) + L) = 0.
\end{equation}
The solution $\lambda = \lambda(i \xi)$ of \eqref{ce} is called the
eigenvalue of the system \eqref{sys1} or \eqref{Fsys1}.
Then we define the notion of the {\it uniform dissipativity} of the
system as follows.

\begin{defn}
The system \eqref{sys1} is called uniformly dissipative
of the type $(p,q)$ if the eigenvalue $\lambda=\lambda(i\xi)$ satisfies
$$
{\rm Re}\,\lambda(i\xi) \leq -c|\xi|^{2p}/(1+|\xi|^2)^q
$$
for $\xi\in\R^n$, where $c$ is a positive constant and $(p,q)$ is a pair
of positive integers.
\end{defn}


For example, under the assumption in Theorem \ref{thm1} or \ref{thm2},
the system \eqref{sys1} is uniformly dissipative of the type $(1,2)$ or
$(1,1)$, respectively.
More precisely, we obtain the following theorem.

\begin{thm}
[Uniform dissipativity]\label{thmUD}
{\rm (i)} Assume the conditions {\rm (A)}, {\rm (S)}, {\rm (S)$_1$} and
{\rm (K)}. Then the system \eqref{sys1} is uniformly dissipative of the type
{\rm (1,2)}.

{\rm (ii)} Assume the conditions {\rm (A)}, {\rm (S)}, {\rm (S)$_2$} and
{\rm (K)}. Then the system \eqref{sys1} is uniformly dissipative of the type
{\rm (1,1)}.
\end{thm}

\begin{proof}
Let $\lambda=\lambda(i\xi)$ be the eigenvalue of the system \eqref{sys1}.
Then we have \eqref{eg} for some $\phi\in{\mathbb C}^m$ with $\phi\neq 0$.
Note that the system \eqref{Fsys1} becomes \eqref{eg} if $d/dt$ and $\hat u$
are replaced by $\lambda$ and $\phi$, respectively.
Therefore, employing the same computations as in the proof of the pointwise
estimate \eqref{loss-point}, we have as a counterpart of \eqref{est-4} that
$$
\{{\rm Re}\,\lambda+c\eta(\xi)\}|\phi|^2\leq 0,
$$
where $\eta(\xi)=|\xi|^2/(1+|\xi|^2)^2$, and $c$ is a positive constant.
Since $\phi\neq 0$, we obtain ${\rm Re}\,\lambda\leq -c\eta(\xi)$,
which proves (i).
Similarly, to prove (ii), the same computations as in the proof of the pointwise estimate
\eqref{std-point} yield the inequality
$\{{\rm Re}\,\lambda+c\rho(\xi)\}|\phi|^2\leq 0$,
where $\rho(\xi)=|\xi|^2/(1+|\xi|^2)$, and $c$ is a positive constant.
This gives ${\rm Re}\,\lambda\leq -c\rho(\xi)$, which proves (ii).
Thus the proof of Theorem \ref{thmUD} is complete.
\end{proof}

Next we discuss the relationship between the conditions (K) and (R).


\begin{thm}
[Relation between (K) and (R)]\label{thmK}
Assume that the condition {\rm (A)} hold.
Then the rank condition {\rm (R)} implies the condition {\rm (K)}, that is,
{\rm (R)} is a sufficient condition for {\rm (K)}.
\end{thm}

\begin{proof}
We assume the conditions (A) and (R). It suffices to construct a
compensating matrix $K(\omega)$. As in \cite{BZ11}, we put
\begin{equation*}
K(\om) = \sum_{k=1}^{m-1} \mu^{\kappa_k}\big\{
   (L \tilde A(\om)^k )^T L \tilde A(\om)^{k-1}
	-(L \tilde A(\om)^{k-1})^T L \tilde A(\om)^{k}\big\}(A^0)^{-1},
\end{equation*}
where $\tilde A(\omega)=(A^0)^{-1}A(\omega)$, $\mu$ is a small positive
constant determined below, and $\kappa_k$ are constants satisfying 
\begin{equation}\label{kappa}
	\begin{split}
&\qquad\qquad\quad 0= \kappa_0< \kappa_1<\cdots <\kappa_m, \\[1mm]
&\kappa_k-(\kappa_{k-1}+\kappa_{k+1})/2\geq\nu \quad \text{for} \quad k=1,2,\cdots, m-1,
	\end{split}
\end{equation}
for some constant $\nu>0$. 
We show that this $K(\om)$ is the desired compensating matrix.
Obviously, we see that $K(-\om) = -K(\om)$ and $(K(\om)A^0)^T = -K(\om)A^0$.
We show that our $K(\omega)$ satisfies \eqref{assump-4*} in the
condition (K). By a simple computation, we have
\begin{equation*}
(K(\om)A(\om))_1
 = \sum_{k=1}^{m-1} \mu^{\kappa_k}(L \tilde A_\om^k )^T L \tilde A_\om^{k}
 - \frac{1}{2} \sum_{k=1}^{m-1} \mu^{\kappa_k}\big\{
    (L \tilde A_\om^{k-1} )^T L \tilde A_\om^{k+1}
	+(L \tilde A_\om^{k+1})^T L \tilde A_\om^{k-1}\big\},
\end{equation*}
where we used the simplified notation $\tilde A_\om = \tilde A(\om)$.
Let $\phi\in {\mathbb C}^m$ and consider the inner product
$\langle (K(\om)A_\om)_1 \phi, \phi \rangle$. It is easy to see that
\begin{equation}\label{innerK}
\langle (K(\om)A(\om))_1 \phi, \phi \rangle
\ge \sum_{k=1}^{m-1} \mu^{\kappa_k}|L \tilde A_\om^{k}\phi|^2
 - \sum_{k=1}^{m-1} \mu^{\kappa_k}
 |L \tilde A_\om^{k-1}\phi| |L \tilde A_\om^{k+1}\phi|. 
\end{equation}
For the second term on the right hand side of \eqref{innerK},
by using \eqref{kappa}, we can estimate as
\begin{equation}\label{thm.ra.p07}
	\begin{split}
\sum_{k=1}^{m-1} \mu^{\kappa_k}
& |L \tilde A_\om^{k-1}\phi| |L \tilde A_\om^{k+1}\phi|
   \leq \mu^{\nu}\sum_{k=1}^{m-1}\mu^{(\kappa_{k-1}+\kappa_{k+1})/2}
   |L \tilde A_\om^{k-1} \phi| |L \tilde A_\om^{k+1} \phi|\\
&\leq \frac{1}{2} \mu^{\nu}\sum_{k=1}^{m-1}
   \big(\mu^{\kappa_{k-1}} |L \tilde A_\om^{k-1} \phi|^2
   +\mu^{\kappa_{k+1}} |L\tilde A_\om^{k+1}\phi|^2\big)
   \leq  \mu^{\nu}\sum_{k=0}^{m} \mu^{\kappa_k} |L \tilde A_\om^{k} \phi|^2,
	\end{split}
\end{equation}
where we assumed $0<\mu<1$.
To estimate the term $|L \tilde A_\om^{m} \phi|^2$,
we consider the characteristic polynomial
\begin{equation*}
    p_m(\la)={\rm det}(\la I - \tilde A_\om)=\la^m+\sum_{k=0}^{m-1}a_k(\om)\la^k,
\end{equation*}
where $a_k(\omega)$ are some polynomials of $\omega\in S^{n-1}$.
Then, by the Cayley-Hamilton theorem, we have $p_m(\tilde A_\om)=0$, that is,
$\tilde A_\om^m=-\sum_{k=0}^{m-1}a_k(\om) \tilde A_\om^k$.
Using this identity, one has
\begin{equation*}
\mu^{\kappa_{m}} |L \tilde A_\om^{m} \phi|^2
   \leq C_1\mu^{\kappa_{m}} \sum_{k=0}^{m-1}|L \tilde A_\om^{k} \phi|^2
   \leq C_1 \sum_{k=0}^{m-1}\mu^{\kappa_k}|L \tilde A_\om^{k} \phi|^2
\end{equation*}
for $0 < \mu <1$, where $C_1$ is a constant satisfying
$|a_k(\om)|^2\leq C_1$ for $k=0,1,\cdots,m-1$ and $\omega\in S^{n-1}$.
%
%
Plugging the above estimate into \eqref{thm.ra.p07} yields
\begin{equation}\label{I1}
\sum_{k=1}^{m-1}\mu^{\kappa_k}
   | L \tilde A_\om^{k-1} \phi| |L \tilde A_\om^{k+1} \phi|
   \leq \mu^{\nu}(1+C_1)\sum_{k=0}^{m-1}\mu^{\kappa_k}|L \tilde A_\om^{k} \phi|^2.
\end{equation}
We substitute \eqref{I1} into \eqref{innerK} to get
\begin{equation*}
\langle (K(\om)A_\om)_1 \phi, \phi \rangle
   \ge \{1 -\mu^\nu(1+C_1)\}\sum_{k=1}^{m-1} \mu^{\kappa_k}|L \tilde A_\om^{k}\phi|^2
   - \mu^{\nu}(1+C_1) |L \phi|^2.
\end{equation*}
Therefore, letting $\mu>0$ suitably small, we obtain
\begin{equation}\label{estK}
\langle (K(\om)A(\om))_1 \phi, \phi \rangle
 \ge c \sum_{k=0}^{m-1} \mu^{\kappa_k}|L \tilde A_\om^{k}\phi|^2
 - C |L \phi|^2,
\end{equation}
where $c$ and $C$ are positive constants.
%
%
Now we use the rank condition (R) and deduce from Lemma \ref{lem.ra} below that
$\sum_{k=0}^{m-1} \mu^{\kappa_k}|L \tilde A_\om^{k}\phi|^2>0$ for each
$\phi\in \C^m$ with $|\phi|=1$ and $\omega\in S^{n-1}$.
Then, by the property of continuous functions on compact sets, we find a
positive constant $c$ such that
$\sum_{k=0}^{m-1} \mu^{\kappa_k}|L \tilde A_\om^{k}\phi|^2\geq c$ for
any $\phi\in \C^m$ with $|\phi|=1$ and $\omega\in S^{n-1}$.
Hence we have
$\sum_{k=0}^{m-1} \mu^{\kappa_k}|L \tilde A_\om^{k}\phi|^2
\geq c|\phi|^2$ for any $\phi\in \C^m$ and $\omega\in S^{n-1}$.
Substituting this inequality into \eqref{estK}, we conclude that
\begin{equation*}
\langle (K(\om)A(\om))_1 \phi, \phi \rangle \ge c |\phi|^2 - C |L \phi|^2,
\end{equation*}
where $c$ and $C$ are positive constants. This shows \eqref{assump-4*} in
the condition (K) and therefore the proof of Theorem \ref{thmK} is complete.
\end{proof}


The rest of this section is devoted to the proof of the following

\begin{lem}\label{lem.ra}
Let $k$ and $m$ be positive integers, and let $M_1,M_2,\cdots,M_k$ be
$m\times m$ real matrices.
Then the following three statements are equivalent to each other.

{\rm (i)} The $km\times m$ real matrix
\begin{equation*}
 \M:=\left[
 \begin{array}{c}
   M_1\\
   M_2\\
   \vdots\\
   M_k
 \end{array}
 \right]
\end{equation*}
has full column rank $m$, that is, ${\rm Rank}\,\M=m$.

{\rm (ii)} 
There exists an integer $j$ with $1\leq j\leq m$ such that $M_j z\neq 0$
for any $z\in \C^m$ with $z\neq 0$.

{\rm (iii)}\quad 
$\displaystyle
\inf_{0\neq z\in \C^m}\frac{\sum_{j=1}^k |M_jz|^2}{|z|^2}>0$.

%
\end{lem}

\begin{proof}
(i)$\Rightarrow$(ii): Suppose that (ii) fails. Then there is a $z\in \C^m$
with $z\neq 0$ such that $M^j z=0$ for all $j=1,2,\cdots,k$.
For this $z\neq 0$, we have $\M z=0$. This implies that the column rank of
$\M$ can not be full, which is a contradiction to (i).

(ii)$\Rightarrow$(iii): Suppose that (iii) fails. Then we have
%
%
\begin{equation*}
   \inf_{|z|=1,\, z\in \C^m} \sum_{j=1}^k |M_jz|^2=0.
\end{equation*}
By the property of continuous functions over compact sets, we find
a $z\in \C^m$ with $|z|=1$ and hence $z\neq 0$ such that
$\sum_{j=1}^k |M_jz|^2=0$. Thus we have $M^j z=0$ for all $j=1,2,\cdots,k$.
This is a contradiction to (ii).

(iii)$\Rightarrow$(i):
Notice that (iii) is equivalent to (iii)$'$: There exists a constant $c>0$
such that
\begin{equation*}
    \sum_{j=1}^k |M_jz|^2\geq c|z|^2
\end{equation*}
for any $z\in \C^m$. Now we assume that there is a $z\in \C^m$ such that
$\M z=0$. Then we have $M^jz=0$ for all $j=1,2,\cdots,k$.
From (iii)$'$ we conclude that $|z|=0$, that is, $z=0$.
This shows that $\M$ is injective and thus it has full column rank $m$.
This completes the proof of Lemma \ref{lem.ra}.
\end{proof}


\section{Decay structure for systems with constraint}

In this section we consider the system \eqref{sys1} with the constraint
condition
\begin{gather} \label{sys2}
\sum_{j=1}^n Q^j u_{x_j} + R u = 0,
\end{gather}
where $Q^j$ and $R$ are $m_1 \times m$ real constant matrices with $m_1 < m$.
Let $\Pi_1$ be the orthogonal projection from $\C^{m_1}$ onto
${\rm Image}(R):=\{R\phi\,;\ \phi \in{\mathbb C}^m\} \subset \C^{m_1}$,
and put $\Pi_2 := I-\Pi_1$.
Notice that $\Pi_1$ and $\Pi_2$ are $m_1 \times m_1$ real symmetric matrices.
By using these projections, we decompose the condition \eqref{sys2} as
\begin{equation*}
\sum_{j=1}^n \Pi_1 Q^j u_{x_j} + R u = 0, 
\qquad\quad
\sum_{j=1}^n  \Pi_2 Q^j u_{x_j}= 0. 
\end{equation*}
We take the Fourier transform of \eqref{sys2}. This yields
\begin{gather}
i |\xi| Q(\omega) \hat u + R \hat u = 0, \label{Fsys2}
\end{gather}
where $Q(\omega) := \sum_{j=1}^n Q^j\omega_j$.
The condition \eqref{Fsys2} is decomposed as
\begin{gather}
i |\xi| \Pi_1 Q(\omega) \hat u + R \hat u = 0, \label{Fsys2-1}
\\[2mm]
i |\xi| \Pi_2 Q(\omega) \hat u = 0. \label{Fsys2-2}
\end{gather}

First we formulate a condition concerning the constraint \eqref{sys2}.

\medskip

\noindent
{\bf Condition(C):}\ \
The matrices $Q(\om)$ and $R$ satisfy
\begin{equation}\label{assump-QR}
\begin{split}
&Q(\omega)(A^0)^{-1}A(\omega) = 0, \qquad
R(A^0)^{-1}L = 0, \\[1mm]
&Q(\omega)(A^0)^{-1}L + R(A^0)^{-1}A(\omega) = 0  
\end{split}
\end{equation}
for each $\omega\in S^{n-1}$.

\medskip

\noindent
This condition (C) implies the following fact:
\eqref{sys2} (or \eqref{Fsys2}) holds at an arbitrary time $t>0$ for the
solution of \eqref{sys1} (or \eqref{Fsys1}) if it holds initially.
Indeed, by differentiating \eqref{Fsys2} with respect to $t$ and
using \eqref{Fsys1}, we obtain
\begin{equation*}
\frac{d}{dt}(i|\xi|Q(\omega)\hat u + R\hat u)
=-(i|\xi|Q(\omega) + R)(A^0)^{-1}(i|\xi|A(\omega) + L)\hat u = 0.
\end{equation*}
%

%
%

Next we formulate new structural conditions which are useful to treat the
Euler-Maxwell system in Section 7. In order to take into account of the
constraint condition \eqref{Fsys2-2}, we introduce the subspace $X_\omega$
of $\C^m$ by
\begin{equation}\label{X}
X_\omega :=\{\phi \in \C^m\,;\ \Pi_2Q(\omega)\phi=0\}.
\end{equation}
Using this subspace, we modify the condition (K) as follows.

\medskip

\noindent
{\bf Condition (K*):}\ \
There is a real matrix $K(\omega) \in C^{\infty}(S^{n-1})$ with the
following properties:
$K(-\omega) = - K(\omega)$, $(K(\omega)A^0)^T = - K(\omega)A^0$ and
\begin{equation}\label{assump-4'*}
(K(\omega)A(\omega))_1 > 0 \quad \text{on}
\quad X_\omega\cap{\rm Ker}(L)
\end{equation}
for each $\omega\in S^{n-1}$, where $X_\omega$ is the subspace defined
in \eqref{X}.

\begin{remark}
Under the conditions {\rm (A)} and {\rm (S)}, the positivity \eqref{assump-4'*}
in the condition {\rm (K*)} holds if and only if
\begin{equation}\label{assump-4'}
\alpha (K(\omega)A(\omega))_1+(SL)_1+L_1>0 \quad \text{on} \quad X_\omega
\end{equation}
for each $\omega\in S^{n-1}$, where $\alpha$ is a suitably small constant.
\end{remark}

The following conditions are modifications of the conditions (S)$_1$ and
(S)$_2$, respectively.

\medskip

\noindent
{\bf Condition (S*)$_1$:}\ \
The matrix $S$ in the condition (S) satisfies
\begin{gather}
i(SA(\omega)-T(\omega))_2\geq 0 \quad \text{on} \quad {\rm Ker}(L_1) \label{assump-3'}
\end{gather}
for each $\omega\in S^{n-1}$, where $T(\omega)$ is the $m\times m$ real
matrix given by $T(\omega) := (\Pi_1Q(\omega))^T\tilde{S}R$ with
$\tilde{S}$ being an $m_1 \times m_1$ real matrix such that
$\tilde{S}_1 \geq 0$ on ${\rm Image}(R)$.

\medskip

\noindent
{\bf Condition (S*)$_2$:}\ \
The matrix $S$ in the condition (S) satisfies
\begin{equation}\label{assump-5'}
i(SA(\omega)-T(\omega))_2 \ge 0 \quad \text{on} \quad \C^m
\end{equation}
for each $\omega\in S^{n-1}$, where $T(\omega)$ is the same matrix as in
the condition (S*)$_1$.

\medskip

Under the above conditions, we obtain the following decay results.

\begin{thm}
[Decay property of the regularity-loss type]\label{thm1'}
Suppose that conditions {\rm (A)}, {\rm (C)}, {\rm (S)}, {\rm (S*)$_1$}
and {\rm (K*)} hold.
Let $s \ge 0$ be an integer and we suppose that the initial data
$u_0$ belong to $H^s \cap L^1$ and satisfy \eqref{sys2}.
Then the solution to the Cauchy problem \eqref{sys1}-\eqref{ID} satisfies
\eqref{sys2} for all $t>0$.
Moreover, the solution satisfies the pointwise estimate \eqref{loss-point}
and decay estimate \eqref{loss-decay} stated in Theorem \ref{thm1}.
\end{thm}

\begin{thm}
[Decay property of the standard type]\label{thm2'}
If the condition {\rm (S*)$_1$} in Theorem \ref{thm1'} is replaced by the
stronger condition {\rm (S*)$_2$}, then the pointwise estimate
\eqref{loss-point} and the decay estimate \eqref{loss-decay} in
Theorem \ref{thm1'} can be improved to \eqref{std-point} and \eqref{std-decay}
stated in Theorem \ref{thm}, respectively.
\end{thm}


\noindent
{\bf Proof of Theorems \ref{thm1'} and \ref{thm2'}.}\
First we observe that the solution $\hat u(t,\xi)$ of the system
\eqref{Fsys1} satisfies the constraint condition \eqref{Fsys2} and hence
\eqref{Fsys2-1} and \eqref{Fsys2-2} for $t>0$ and $\xi\in \R^n$.
In particular, we have
\begin{equation}\label{constraint}
\hat u(t,\xi)\in X_\omega
\end{equation}
for $t>0$ and $\xi\in \R^n$, where $X_\omega$ is the subspace defined in
\eqref{X}.

We show the pointwise estimate \eqref{loss-point}. Then the corresponding
decay estimate \eqref{loss-decay} can be shown just in the same way as
before. We employ the same computations as in the proof of Theorem
\eqref{thm1} and obtain the energy equality \eqref{energy}.
This energy equality is rewritten as
\begin{equation}\label{energy'}
\frac{1}{2}\frac{d}{dt}E+ D_1 + D_2' = G',
\end{equation}
where $E$ and $D_1$ are defined in \eqref{note-energy} and
\begin{equation*}
\begin{split}
&(1+|\xi|^2)^2 D_2' := \alpha_1 |\xi|(1+|\xi|^2)
 \big\{\langle i(SA(\omega)-T(\omega))_2 P_1 \hat u, P_1 \hat u \rangle
 + \langle i(T(\omega))_2 \hat u, \hat u \rangle\big\}, \\[1mm]
&(1+|\xi|^2)^2 G'
   := \alpha_1 \alpha_2|\xi| \langle i\alpha(K(\omega)L)_2 \hat u, \hat u \rangle \\
&\qquad
 - \alpha_1 |\xi|(1+|\xi|^2) \big\{
 \langle i(SA(\omega)-T(\omega))_2 \hat u, \hat u \rangle
   - \langle i(SA(\omega)-T(\omega))_2 P_1 \hat u, P_1 \hat u \rangle \big\}.
\end{split}
\end{equation*}
Here the term $E_1$ was estimated in \eqref{est-1} for suitably small
$\alpha_1>0$ and $\alpha_2>0$.
Also we note that the term $D_1$ has the expression \eqref{dissipation}.
Since our solution verifies \eqref{constraint}, we can use the positivity
\eqref{assump-4'} which is based on the condition (K*) and conclude that
$D_1$ satisfies the same estimate \eqref{est-2} for suitably small
$\alpha_1>0$ and $\alpha_2>0$.
%
%
%
Next we treat the term $D_2'$.
By virtue of \eqref{assump-3'} in the condition (S*)$_1$, we have
$\langle i(SA(\omega)-T(\omega))_2 P_1 \hat u, P_1 \hat u \rangle \geq 0$.
On the other hand, using the explicit form of the matrix $T(\omega)$ in
(S*)$_1$, we see that
\begin{equation*}
\begin{split}
&|\xi| \langle i(T(\omega))_2 \hat u, \hat u \rangle
= |\xi| \langle i( (\Pi_1Q(\omega))^T\tilde{S}R)_2 \hat u, \hat u \rangle \\
&\qquad
= \frac{1}{2}i|\xi| \big\{ \langle (\Pi_1Q(\omega))^T\tilde{S}R \hat u, \hat u \rangle
   -  \langle \hat u, (\Pi_1Q(\omega))^T\tilde{S}R \hat u \rangle \big\}.
\end{split}
\end{equation*}
Moreover, using the constraint \eqref{Fsys2-1}, we know that
$$
i|\xi|\langle(\Pi_1Q(\omega))^T\tilde{S}R\hat u,\hat u\rangle
=i|\xi|\langle\tilde{S}R\hat u,\Pi_1Q(\omega)\hat u\rangle
=\langle\tilde{S}R\hat u,R\hat u \rangle.
$$
Similarly, we have
$i|\xi| \langle\hat u,(\Pi_1Q(\omega))^T\tilde{S}R\hat u\rangle
=-\langle\tilde{S}^TR\hat u,R\hat u\rangle$.
Consequently, we find that
$$
|\xi| \langle i(T(\omega))_2 \hat u, \hat u \rangle
=\langle\tilde{S}_1R\hat u,R\hat u \rangle.
$$
Hence we obtain
\begin{equation}\label{D2'}
(1+|\xi|^2)^2D_2'
\geq \alpha_1(1+|\xi|^2)\langle\tilde{S}_1R\hat u,R\hat u \rangle
\geq 0,
\end{equation}
where we used the nonnegativity of $\tilde S_1$ on ${\rm Image}(R)$
in the last inequality.
%
%
Finally, we estimate the term $G'$.
For the first term in $G'$, we have the estimate \eqref{est-3a}.
Also, similarly to \eqref{est-3c}, we have
\begin{equation*}
\begin{split}
&|\xi|(1+|\xi|^2)
   \big|\langle i(SA(\omega)-T(\omega))_2 \hat u, \hat u \rangle -
\langle i(SA(\omega)-T(\omega))_2 P_1 \hat u, P_1\hat u \rangle\big| \\[1mm]
&\leq \delta|\xi|^2|\hat u|^2
   +C_\delta(1+|\xi|^2)^2|(I-P_1)\hat u|^2
\end{split}
\end{equation*}
for any $\delta>0$, where $C_\delta$ is a constant depending on $\delta$.
Thus, as a counterpart of \eqref{est-G}, we obtain
\begin{equation*}
\begin{split}
(1+|\xi|^2)^2|G'|
&\leq \alpha_1(\alpha_2\epsilon+\delta)|\xi|^2|\hat u|^2 \\[1mm]
&+\alpha_1\alpha_2C_\epsilon|(I-P)\hat u|^2
+\alpha_1C_\delta(1+|\xi|^2)^2|(I-P_1)\hat u|^2.
\end{split}
\end{equation*}

Now we choose $\epsilon$, $\delta$, $\alpha_1$ and $\alpha_2$ suitably small
as in the proof of Theorem \ref{thm1}, and deduce that
$|G'| \le D_1/2$, where $D_1$ satisfies \eqref{est-2a} by \eqref{est-2}.
Consequently, \eqref{energy'} becomes
\begin{equation}\label{energy2'}
\frac{d}{dt}E + D_1 + 2D_2' 
\leq 0.
\end{equation}
%
%
Since $D_1\geq c\eta(\xi)E$ and $D_2'\geq 0$ by \eqref{est-1},
\eqref{est-2a} and \eqref{D2'}, the inequality \eqref{energy2'} is reduced to
$\frac{d}{dt}E+c\eta(\xi)E\leq 0$, where
$\eta(\xi)=|\xi|^2/(1+|\xi|^2)^2$, and $c$ is a positive constant.
Solving this differential inequality and using \eqref{est-1},
we arrive at the desired pointwise estimate \eqref{loss-point}.
Thus the proof of Theorem \ref{thm1'} is complete.

Finally, we prove the pointwise estimate \eqref{std-point}.
To this end, we rewrite the energy equality \eqref{energy-2} in the form
\begin{equation*}
\frac{1}{2}\frac{d}{dt}\tilde E+ \tilde D_1 + \tilde D_2' = \tilde G,
\end{equation*}
where $\tilde E$, $\tilde D_1$ and $\tilde G$ are defined in \eqref{ene} and
\begin{equation*}
\tilde D_2' := \alpha_1 |\xi|
 \big\{\langle i(SA(\omega)-T(\omega))_2 \hat u, \hat u \rangle
 + \langle i(T(\omega))_2 \hat u, \hat u \rangle\big\}.
\end{equation*}
Here, using \eqref{assump-5'} in the condition (S*)$_2$ and computing
similarly as in the derivation of \eqref{D2'}, we have
\begin{equation*}
D_2'\geq \alpha_1\langle\tilde{S}_1R\hat u,R\hat u \rangle
\geq 0.
\end{equation*}
On the other hand, the previous estimates for $\tilde E$, $\tilde D_1$ and
$\tilde G$ are valid also in the present case. Therefore, by employing the
same computing as in the proof of Theorem \ref{thm2}, we can deduce the
desired pointwise estimate \eqref{std-point}.
This completes the proof of Theorem \ref{thm2'}.
\qed


\section{Application to the Timoshenko system}

In this section, as an application of Theorems \ref{thm1} and \ref{thm2}, 
we treat the following dissipative Timoshenko system
\begin{equation}\label{Timo}
\left\{
\begin{split}
& w_{tt} - (w_x - \psi)_x =0, \\
& \psi_{tt} - a^2 \psi_{xx} - (w_x - \psi) + \gamma \psi_t = 0,
\end{split}
\right.
\end{equation}
where $a$ and $\gamma$ are positive constants,
and $w=w(t,x)$ and $\psi=\psi(t,x)$ are unknown scalar functions of $t>0$
and $x \in \mathbb{R}$. The Timoshenko system above is a model system describing 
the vibration of the  beam called the Timoshenko beam, and 
$w$ and $\psi$ denote the transversal displacement and the rotation angle 
of the beam, respectively. Here we only mention \cite{ABMR,MRR} and 
\cite{IHK08,IK08,LK4} for related mathematical results. 

As in \cite{IHK08,IK08}, we introduce the vector-valued function 
$u=(w_x-\psi, w_t, a\psi_x,\psi_t)^T$.
Then the Timoshenko system \eqref{Timo} is written in the form of
\eqref{sys1} with the coefficient matrices
\begin{equation}\label{Tmat}
A^0=I, \qquad
A =
-\left(
\begin{array}{cccc}
 {0} & {1} & {0} & {0} \\
 {1} & {0} & {0} & {0} \\
 {0} & {0} & {0} & {a} \\
 {0} & {0} & {a} & {0} \\
\end{array}
\right), \qquad
L =
\left(
\begin{array}{cccc}
 {0} & {0} & {0} & {1} \\
 {0} & {0} & {0} & {0} \\
 {0} & {0} & {0} & {0} \\
 {-1} & {0} & {0} & {\gamma} \\
\end{array}
\right),
\end{equation}
where 
$I$ is the $4\times 4$ identity matrix.
Here the space dimension is $n=1$ and the size of the system is $m=4$.
Notice that the relaxation matrix $L$ is not symmetric.
For this Timoshenko system we obtain the following result.

\begin{thm}
[Decay property for the Timoshenko system]\label{thmT}
The Timoshenko system with $a>0$ {\rm (}resp. $a=1${\rm )} satisfies
all the conditions in Theorem \ref{thm1}
{\rm (}resp. Theorem \ref{thm2}{\rm )}.
Therefore the solution to the Timoshenko system with $a>0$
{\rm (}resp. $a=1${\rm )}
verifies the pointwise estimate \eqref{loss-point}
{\rm (}resp. \eqref{std-point}{\rm )}
and the decay estimate \eqref{loss-decay}
{\rm (}resp. \eqref{std-decay}{\rm )}.
\end{thm}

\begin{proof}
The symmetric part of $L$ is given by
\begin{equation*}
L_1 =
\left(
\begin{array}{cccc}
 {0} & {0} & {0} & {0} \\
 {0} & {0} & {0} & {0} \\
 {0} & {0} & {0} & {0} \\
 {0} & {0} & {0} & {\gamma} \\
\end{array}
\right)
\end{equation*}
and we see that
\begin{equation*}
{\rm Ker}(L)={\rm span}\{e_2,e_3\}, \qquad
{\rm Ker}(L_1)={\rm span}\{e_1,e_2,e_3\},
\end{equation*}
where $e_1=(1,0,0,0)^T$, $e_2=(0,1,0,0)^T$, $e_3=(0,0,1,0)^T$,
and $e_4=(0,0,0,1)^T$.
It is obvious that the matrices in \eqref{Tmat} satisfies the condition (A).
For example, we have $\langle L\phi,\phi\rangle=\gamma|\phi_4|^2\geq 0$
for $\phi=(\phi_1,\phi_2,\phi_3,\phi_4)^T\in\C^4$, so that $L\geq 0$ on
$\C^4$.

%

We verify the conditions (K), (S) and (S)$_1$ for $a>0$, and also the
condition (S)$_2$ for $a=1$. To this end, we define the real symmetric
matrix $S$ and the real skew-symmetric matrix $K$ by
\begin{equation}\label{KS}
S = - \beta
\left(
\begin{array}{cccc}
 {0} & {0} & {0} & {1} \\
 {0} & {0} & {a} & {0} \\
 {0} & {a} & {0} & {0} \\
 {1} & {0} & {0} & {0} \\
\end{array}
\right), \qquad
K = 
\left(
\begin{array}{cccc}
 {0} & {1} & {0} & {0} \\
 {-1} & {0} & {0} & {0} \\
 {0} & {0} & {0} & {-1} \\
 {0} & {0} & {1} & {0} \\
\end{array}
\right),
\end{equation}
where $\beta$ 
is a positive constant determined later.
This choice of the matrices $S$ and $K$ is based on the computations
employed in \cite{IHK08,IK08}.
A simple computation, using \eqref{Tmat} and \eqref{KS}, yields
\begin{equation*}
	\begin{split}
&SA = \beta
\left(
\begin{array}{cccc}
 {0} & {0} & {a} & {0} \\
 {0} & {0} & {0} & {a^2} \\
 {a} & {0} & {0} & {0} \\
 {0} & {1} & {0} & {0} \\
\end{array}
\right),  \qquad
SL = \beta
\left(
\begin{array}{cccc}
 {1} & {0} & {0} & {-\gamma} \\
 {0} & {0} & {0} & {0} \\
 {0} & {0} & {0} & {0} \\
 {0} & {0} & {0} & {-1} \\
\end{array}
\right), \qquad
KA = 
\left(
\begin{array}{cccc}
 {-1} & {0} & {0} & {0} \\
 {0} & {1} & {0} & {0} \\
 {0} & {0} & {a} & {0} \\
 {0} & {0} & {0} & {-a} \\
\end{array}
\right).
\end{split}
\end{equation*}
Hence we have
\begin{equation*}
(SA)_2 = \frac{1}{2}\beta(a^2-1)
\left(
\begin{array}{cccc}
 {0} & {0} & {0} & {0} \\
 {0} & {0} & {0} & {1} \\
 {0} & {0} & {0} & {0} \\
 {0} & {-1} & {0} & {0} \\
\end{array}
\right), \qquad
(SL)_1 = \beta
\left(
\begin{array}{cccc}
 {1} & {0} & {0} & {-\gamma/2} \\
 {0} & {0} & {0} & {0} \\
 {0} & {0} & {0} & {0} \\
 {-\gamma/2} & {0} & {0} & {-1} \\
\end{array}
\right),
\end{equation*}
and $(KA)_1 = KA$.

First we check the condition (K). A simple computation gives
\begin{equation*}
\langle(KA)_1\phi,\phi\rangle
=-|\phi_1|^2+|\phi_2|^2+a|\phi_3|^2-a|\phi_4|^2
\end{equation*}
for $\phi=(\phi_1,\phi_2,\phi_3,\phi_4)^T\in\C^4$.
Let $\phi\in{\rm Ker}(L)$. Then $\phi=(0,\phi_2,\phi_3,0)^T$.
For this $\phi$, we have
\begin{equation*}
\langle(KA)_1\phi,\phi\rangle=|\phi_2|^2+a|\phi_3|^2.
\end{equation*}
This shows \eqref{assump-4*} and hence the condition (K) has been verified.
Next we check the condition (S). We have
\begin{equation*}
(SL)_1+L_1 =
\left(
\begin{array}{cccc}
 {\beta} & {0} & {0} & {-\beta \gamma/2} \\
 {0} & {0} & {0} & {0} \\
 {0} & {0} & {0} & {0} \\
 {-\beta \gamma/2} & {0} & {0} & {\gamma-\beta} \\
\end{array}
\right). 
\end{equation*}
Then a simple computation gives
\begin{equation*}
\begin{split}
\langle ((SL)_1+L_1)\phi,\phi \rangle
&=\beta |\phi_1|^2 + (\gamma-\beta)|\phi_4|^2
   -\beta \gamma {\rm Re}(\phi_1 \bar \phi_4) \\[1mm]
&\geq\beta |\phi_1|^2 + (\gamma-\beta)|\phi_4|^2
   -\beta \gamma |\phi_1||\bar \phi_4|
\end{split}
\end{equation*}
for $\phi=(\phi_1,\phi_2,\phi_3,\phi_4)^T\in\C^4$.
The corresponding discriminant is
$\beta^2 \gamma^2 -4\beta(\gamma-\beta)
= \beta\{(\gamma^2+4)\beta -4\gamma\}$.
Therefore, letting $\beta>0$ so small that $\beta < 4\gamma/(\gamma^2+4)$,
we get
\begin{equation*}
\langle ((SL)_1+L_1)\phi,\phi \rangle
\ge c (|\phi_1|^2 + |\phi_4|^2),
\end{equation*}
where $c$ is a positive constant. This shows that
$(SL)_1+L_1\geq 0$ on $\C^4$ and
${\rm Ker}((SL)_1+L_1)={\rm span}\{e_2,e_3\}$.
Hence we have ${\rm Ker}((SL)_1+L_1)={\rm Ker}(L)$.
Thus we have verified the condition (S).
Finally, we check \eqref{assump-3} in the condition (S)$_1$.
By direct calculation, we get
\begin{equation*}
\langle i (SA)_2\phi,\phi \rangle
= \beta (a^2-1) {\rm Im}(\phi_2 \bar \phi_4)
\end{equation*}
for $\phi = (\phi_1,\phi_2,\phi_3,\phi_4)^T\in{\mathbb C}^4$.
Let $\phi \in {\rm Ker}(L_1)$.
Then $\phi = (\phi_1,\phi_2,\phi_3,0)$. For this $\phi$, we have
$\langle i (SA)_2\phi,\phi \rangle = 0$.
This shows \eqref{assump-3} and hence the condition (S)$_1$ has been
verified.
Consequently, Theorem \ref{thm1} is applicable to the Timoshenko system
with $a>0$ and we obtain the estimates \eqref{loss-point} and
\eqref{loss-decay}.

%
%
%
%
%
%
%

On the other hand, when $a=1$, we have $(SA)_2=0$, which shows
\eqref{assump-3} in the condition (S)$_2$.
Therefore Theorem \ref{thm2} is applicable to the Timoshenko system with
$a=1$ and we obtain the estimate \eqref{std-point} and \eqref{std-decay}
in this special case. This completes the proof of Theorem \ref{thmT}.
\end{proof}

Finally in this section, we check that the Timoshenko system 
satisfies the condition (R).
By direct calculation, we have
\begin{equation*}
LA =
\begin{pmatrix}
0 & 0 & -a & 0 \\[1mm]
0 & 0 & 0 & 0 \\[1mm]
0 & 0 & 0 & 0 \\[1mm]
0 & 1 & -\gamma a & 0
\end{pmatrix},
\quad
LA^2 =
\begin{pmatrix}
0 & 0 & 0 & a^2 \\[1mm]
0 & 0 & 0 & 0 \\[1mm]
0 & 0 & 0 & 0 \\[1mm]
-1 & 0 & 0 & \gamma a^2
\end{pmatrix},
\quad
LA^3 =
\begin{pmatrix}
0 & 0 & -a^3 & 0 \\[1mm]
0 & 0 & 0 & 0 \\[1mm]
0 & 0 & 0 & 0 \\[1mm]
0 & 1 & -\gamma a^3 & 0
\end{pmatrix}.
\end{equation*}
Moreover, one can verify that the linear system of equations 
$LA^k\phi=0$ $(0\leq k\leq 3)$ has a unique solution $\phi=0$, which implies 
the rank equality \eqref{thm.ra.c4} with $m=4$.
Thus we find that the Timoshenko system satisfies the condition (R).
It means that Theorem \ref{thm1} and \ref{thm2} with condition (K) replaced by 
the condition (R) are applicable to the Timoshenko system.


\section{Application to the Euler-Maxwell system}

In this last section,  as an
application of Theorem \ref{thm1'}, we deal with the following Euler-Maxwell system
\begin{equation}\label{EM1}
\left\{
\begin{split}
& \rho_t + {\rm div} (\rho v) = 0, \\
& (\rho v)_t + {\rm div}(\rho v \otimes v)
   + \nabla p(\rho) = -\rho(E + v \times B) - \rho v, \\
& E_t - {\rm rot}\, B = \rho v, \\
& B_t + {\rm rot}\, E = 0,
\end{split}
\right.
\end{equation}
\begin{equation}\label{EM2}
{\rm div}\, E = \rho_\infty - \rho, \qquad \quad
%
{\rm div}\, B = 0.
\end{equation}
Here the density $\rho>0$,
the velocity $v\in{\mathbb R}^3$, 
the electric field $E\in{\mathbb R}^3$, 
and the magnetic induction $B\in{\mathbb R}^3$ 
are unknown functions of $t>0$ and $x\in{\mathbb R}^3$,
the pressure $p(\rho)$ is a given smooth function of $\rho$ satisfying
$p'(\rho)>0$ for $\rho>0$, 
and $\rho_\infty$ is a positive constant. The Euler-Maxwell system above arises from the study of plasma phsyics; refer to \cite{BCD} for detailed discussions on this model.

We now observe that the system \eqref{EM1} is written in the form of a
symmetric hyperbolic system.
For this purpose, it is convenient to introduce
\begin{equation*}
u=(\rho,v,E,B)^T, \qquad
u_\infty=(\rho_\infty,0,0,B_\infty)^T, 
\end{equation*}
which are regarded as column vectors in ${\mathbb R}^{10}$,
where $B_\infty \in \mathbb{R}^3$ is an arbitrarily fixed constant.
Then the Euler-Maxwell system \eqref{EM1} is rewritten as
\begin{gather}
A^0(u) u_t + \sum_{j=1}^3 A^j(u) u_{x_j} + L(u)u = 0, \label{nsys1}
\end{gather}
where the coefficient matrices are given explicitly as
\begin{equation*}
\begin{split}
&A^0(u)=
\begin{pmatrix}
p'(\rho)/\rho & 0 & 0 & 0 \\[1mm]
0 & \rho I & 0 & 0 \\[1mm]
0 & 0 & I & 0 \\[1mm]
0 & 0 & 0 & I
\end{pmatrix},
\quad
L(u)=
\begin{pmatrix}
0 & 0 & 0 & 0 \\[1mm]
0 & \rho(I-\Omega_B) & \rho I & 0 \\[1mm]
0 & -\rho I & 0 & 0 \\[1mm]
0 & 0 & 0 & 0
\end{pmatrix},
\\[3mm]
&\sum_{j=1}^3A^j(u)\xi_j=
\begin{pmatrix}
(p'(\rho)/\rho)(v\cdot\xi) & p'(\rho)\xi & 0 & 0 \\[1mm]
p'(\rho)\xi^T & \rho(v\cdot\xi)I & 0 & 0 \\[1mm]
0 & 0 & 0 & -\Omega_\xi \\[1mm]
0 & 0 & \Omega_\xi & 0
\end{pmatrix}.
\end{split}
\end{equation*}
Here $I$ denotes the $3\times 3$ identity matrix,
$\xi=(\xi_1,\xi_2,\xi_3)\in{\mathbb R}^3$, and
$\Omega_\xi$ is the skew-symmetric matrix defined by
\begin{equation*}
\Omega_\xi=
\begin{pmatrix}
0 & -\xi_3 & \xi_2 \\[1mm]
\xi_3 & 0 & -\xi_1 \\[1mm]
-\xi_2 & \xi_1 & 0
\end{pmatrix}
\end{equation*}
for $\xi=(\xi_1,\xi_2,\xi_3)\in{\mathbb R}^3$, so that we have
$\Omega_\xi E^T=(\xi\times E)^T$ (as a column vector in
${\mathbb R}^3$) for $E=(E_1,E_2,E_3)\in{\mathbb R}^3$.
We note that \eqref{nsys1} is a symmetric hyperbolic system because
$A^0(u)$ is real symmetric and positive definite and $A^j(u)$, $j=1,2,3$,
are real symmetric.
Also, the matrix $L(u)$ is nonnegative definite, so that it is
regarded as a relaxation matrix. Moreover, we have
$L(u)u_\infty=0$ for each $u$ so that the constant state $u_\infty$
lies in the kernel of $L(u)$.
However, the matrix $L(u)$ or $L(u_\infty)$ has skew-symmetric part
and is not real symmetric.
Consequently, our system is not included in a class of systems
considered in \cite{UKS84,SK85}.


The constant state $u_\infty$ is an equilibrium of the system \eqref{nsys1}
with the constraint \eqref{EM2}.
We consider the linearization of \eqref{nsys1} with \eqref{EM2} around the
equilibrium state $u_\infty$. If we denote $u-u_\infty$ by $u$ again, then
the linearization of the system \eqref{nsys1} with \eqref{EM2} can be
written in the form of \eqref{sys1} with \eqref{sys2}, where the coefficient
matrices are given by
%
%
%
%
%
\begin{equation}\label{mat1}
	\begin{split}
& A^0 =
\left(
\begin{array}{cccc}
 {a_\infty} & {0} & {0} & {0} \\
 {0} & {\rho_\infty I} & {0} & {0} \\
 {0} & {0} & {I} & {0} \\
 {0} & {0} & {0} & {I} \\
\end{array}
\right),  \qquad
A(\xi):=\sum_{j=1}^3A^j \xi_j =
\left(
\begin{array}{cccc}
 {0} & {b_\infty \xi} & {0} & {0} \\
 {b_\infty \xi^T} & {0} & {0} & {0} \\
 {0} & {0} & {0} & {-\Omega_\xi} \\
 {0} & {0} & {\Omega_\xi} & {0} \\
\end{array}
\right), \\[3mm]
&L =
\left(
\begin{array}{cccc}
 {0} & {0} & {0} & {0} \\
 {0} & {\rho_\infty(I- \Omega_{B_\infty})} & {\rho_\infty I} & {0} \\
 {0} & {-\rho_\infty I} & {0} & {0} \\
 {0} & {0} & {0} & {0} \\
\end{array}
\right),
\end{split}
\end{equation}
and
\begin{equation}\label{mat2}
Q(\xi):=\sum_{j=1}^3 Q^j \xi_j =
\left(
\begin{array}{cccc}
 {0} & {0} & {\xi} & {0} \\
 {0} & {0} & {0} & {\xi} \\
\end{array}
\right), \qquad
R =
\left(
\begin{array}{cccc}
 {1} & {0} & {0} & {0} \\
 {0} & {0} & {0} & {0} \\
\end{array}
\right),
\end{equation}
%
%
where $a_\infty = p'(\rho_\infty)/\rho_\infty$ and
$b_\infty = p'(\rho_\infty)$ are positive constants.
Here the space dimension is $n=3$ and the size of the systems are $m=10$
and $m_1 = 2$.
For this linearized Euler-Maxwell system, we obtain the following result.

\begin{thm}
[Decay property for the Euler Maxwell system]\label{thmEM}
The linearized Euler-Maxwell system satisfies all the conditions in Theorem
\ref{thm1'} and therefore the corresponding solution verifies the
pointwise estimate \eqref{loss-point} and the decay estimate
\eqref{loss-decay}.
\end{thm}

\begin{proof}
The symmetric part of $L$ is given by
\begin{equation*}
L_1 =
\left(
\begin{array}{cccc}
 {0} & {0} & {0} & {0} \\
 {0} & {\rho_\infty I} & {0} & {0} \\
 {0} & {0} & {0} & {0} \\
 {0} & {0} & {0} & {0} \\
\end{array}
\right)
\end{equation*}
and we see that
\begin{equation*}
{\rm ker}(L)={\rm span}\{e_1,e_8,e_9,e_{10}\}, \qquad
{\rm ker}(L_1)={\rm span}\{e_1,e_5,e_6,e_7,e_8,e_9,e_{10}\},
\end{equation*}
where $e_1=(1,0,0,0,0,0,0,0,0,0)^T$, 
$\cdots$, $e_{10}=(0,0,0,0,0,0,0,0,0,1)^T$ form the standard orthonormal
basis of $\C^{10}$.
The image of the matrix $R$ in \eqref{mat2} is spanned by $(1,0)^T\in\C^2$.
Therefore the corresponding orthogonal projections $\Pi_1$ and $\Pi_2$
are given respectively by
\begin{equation*}
\Pi_1 =
\left(
\begin{array}{cccc}
 {1} & {0} \\
 {0} & {0} \\
\end{array}
\right), \qquad
\Pi_2 =
\left(
\begin{array}{cccc}
 {0} & {0} \\
 {0} & {1} \\
\end{array}
\right).
\end{equation*}
Since
\begin{equation*}
\Pi_2 Q(\omega) =
\left(
\begin{array}{cccc}
 {0} & {0} & {0} & {0} \\
 {0} & {0} & {0} & {\omega} \\
\end{array}
\right),
\end{equation*}
the subspace $X_\omega$ defined in \eqref{X} consists of vectors
$\phi=(\phi_1,\phi_2,\phi_3,\phi_4)\in\C^{10}$ such that 
$\phi_1\in\C$, $\phi_2,\,\phi_3,\,\phi_4\in\C^3$ and 
$\omega\cdot\phi_4=0$.

It is easy to check that the matrices in \eqref{mat1} satisfy the condition
(A). For instance, we have
$\langle L\phi,\phi\rangle=\rho_\infty|\phi_2|^2\geq 0$
for $\phi=(\phi_1,\phi_2,\phi_3,\phi_4)\in\C^{10}$, where $\phi_1\in\C$
and $\phi_2,\,\phi_3,\,\phi_4\in\C^3$. Thus we see that $L\geq 0$ on
$\C^{10}$.
Also we can check \eqref{assump-QR} in the condition (C) by direct
computations using the expressions
\begin{equation*}
(A^0)^{-1}A(\omega) =
\left(
\begin{array}{cccc}
 {0} & {\rho_\infty \omega} & {0} & {0} \\
 {a_\infty \omega^T} & {0} & {0} & {0} \\
 {0} & {0} & {0} & {-\Omega_\omega} \\
 {0} & {0} & {\Omega_\omega} & {0} \\
\end{array}
\right), \quad
(A^0)^{-1}L =
\left(
\begin{array}{cccc}
 {0} & {0} & {0} & {0} \\
 {0} & {I-\Omega_{B_\infty}} & {I} & {0} \\
 {0} & {- \rho_\infty I} & {0} & {0} \\
 {0} & {0} & {0} & {0} \\
\end{array}
\right).
\end{equation*}

%

We show that our Euler-Maxwell system satisfies the conditions (K*), (S)
and (S*)$_1$. We define the real matrices $S$ and $K(\omega)$ by
\begin{equation*}
S = \beta
\left(
\begin{array}{cccc}
 {0} & {0} & {0} & {0} \\
 {0} & {0} & {I} & {0} \\
 {0} & {(1/\rho_\infty)I} & {0} & {0} \\
 {0} & {0} & {0} & {0} \\
\end{array}
\right), \qquad
K(\omega) = 
\left(
\begin{array}{cccc}
 {0} & {(1/\rho_\infty)\omega} & {0} & {0} \\
 {-(1/a_\infty)\omega^T} & {0} & {0} & {0} \\
 {0} & {0} & {0} & {\Omega_\omega} \\
 {0} & {0} & {\Omega_\omega} & {0} \\
\end{array}
\right),
\end{equation*}
where $\beta$ 
is a positive constant determined later.
This choice of $S$ and $K(\omega)$ is based on the computations employed
in our previous papers \cite{D,USK,UK}.
Then straightforward computations yield
\begin{equation*}
	\begin{split}
&SA^0 = \beta
\left(
\begin{array}{cccc}
 {0} & {0} & {0} & {0} \\
 {0} & {0} & {I} & {0} \\
 {0} & {I} & {0} & {0} \\
 {0} & {0} & {0} & {0} \\
\end{array}
\right),  \ \
SA(\omega) = \beta
\left(
\begin{array}{cccc}
 {0} & {0} & {0} & {0} \\
 {0} & {0} & {0} & {-\Omega_\omega} \\
 {a_\infty \omega^{T}} & {0} & {0} & {0} \\
 {0} & {0} & {0} & {0} \\
\end{array}
\right), \ \ 
SL = \beta
\left(
\begin{array}{cccc}
 {0} & {0} & {0} & {0} \\
 {0} & {-\rho_\infty I} & {0} & {0} \\
 {0} & {I - \Omega_{B_\infty}} & {I} & {0} \\
 {0} & {0} & {0} & {0} \\
\end{array}
\right),  \\[3mm]
&
K(\omega)A^0 = 
\left(
\begin{array}{cccc}
 {0} & {\omega} & {0} & {0} \\
 {-\omega^T} & {0} & {0} & {0} \\
 {0} & {0} & {0} & {\Omega_\omega} \\
 {0} & {0} & {\Omega_\omega} & {0} \\
\end{array}
\right), \quad
K(\omega)A(\omega) = 
\left(
\begin{array}{cccc}
 {a_\infty} & {0} & {0} & {0} \\
 {0} & {-\rho_\infty(\omega \otimes \omega)} & {0} & {0} \\
 {0} & {0} & {\Omega_\omega^2} & {0} \\
 {0} & {0} & {0} & {-\Omega_\omega^2} \\
\end{array}
\right).
\end{split}
\end{equation*}
Hence we see that
\begin{equation*}
\begin{split}
(SA(\omega))_2 &= \frac{1}{2}\beta
\left(
\begin{array}{cccc}
 {0} & {0} & {-a_\infty \omega} & {0} \\
 {0} & {0} & {0} & {-\Omega_\omega} \\
 {a_\infty \omega^{T}} & {0} & {0} & {0} \\
 {0} & {-\Omega_\omega} & {0} & {0} \\
\end{array}
\right), \quad 
(SL)_1 = \beta
\left(
\begin{array}{cccc}
 {0} & {0} & {0} & {0} \\
 {0} & {-\rho_\infty I} & {\frac{1}{2}(I + \Omega_{B_\infty})} & {0} \\
 {0} & {\frac{1}{2}(I - \Omega_{B_\infty})} & {I} & {0} \\
 {0} & {0} & {0} & {0} \\
\end{array}
\right),
\end{split}
\end{equation*}
and $(K(\omega)A(\omega))_1 = K(\omega)A(\omega)$.

First we check the condition (K*).
Obviously we see that $K(-\omega)=-K(\omega)$ and
$(K(\omega)A(\omega))^T=-K(\omega)A(\omega)$.
Also a simple computation gives
\begin{equation*}
\langle ((K(\omega)A(\omega))_1\phi,\phi \rangle
= a_\infty |\phi_1|^2 - \rho_\infty|\omega \cdot \phi_2|^2
- |\Omega_\omega \phi_3|^2 + |\Omega_\omega \phi_4|^2
\end{equation*}
for $\phi=(\phi_1,\phi_2,\phi_3,\phi_4)\in\C^{10}$, where $\phi_1\in\C$
and $\phi_2,\,\phi_3,\,\phi_4\in\C^3$.
Now we suppose that $\phi\in X_\omega\cap{\rm Ker}(L)$. Then
$\phi=(\phi_1,0,0,\phi_4)$ with $\omega\cdot\phi_4=0$.
For this $\phi$, we have $|\Omega_\omega \phi_4|^2=|\phi_4|^2$ and hence
\begin{equation*}
\langle ((K(\omega)A(\omega))_1\phi,\phi \rangle
= a_\infty |\phi_1|^2 + |\phi_4|^2.
\end{equation*}
This shows \eqref{assump-4'*}. Therefore we have checked the condition
(K*).

Next we check the condition (S). We have
\begin{equation*}
(SL)_1+L_1 =
\left(
\begin{array}{cccc}
 {0} & {0} & {0} & {0} \\
 {0} & {(1 -\beta) \rho_\infty I} & {\frac{1}{2}\beta(I + \Omega_{B_\infty})} & {0} \\
 {0} & {\frac{1}{2}\beta(I - \Omega_{B_\infty})} & {\beta I} & {0} \\
 {0} & {0} & {0} & {0} \\
\end{array}
\right). 
\end{equation*}
Then a simple computation gives
\begin{equation*}
\begin{split}
\langle ((SL)_1+L_1)\phi,\phi \rangle
&= (1-\beta) \rho_\infty |\phi_2|^2 + \beta |\phi_3|^2
   + \beta {\rm Re}\{(I+\Omega_{B_\infty})\phi_3 \cdot \bar \phi_2\} \\[1mm]
&\geq (1-\beta) \rho_\infty |\phi_2|^2 + \beta |\phi_3|^2
   - \beta(1+|B_\infty|)|\phi_2||\phi_3|
\end{split}
\end{equation*}
for $\phi = (\phi_1,\phi_2,\phi_3,\phi_4)^T\in{\mathbb C}^{10}$.
The corresponding discriminant is
$\beta^2(1+|B_\infty|)^2 - 4\beta(1-\beta)\rho_\infty
= \beta\{(4\rho_\infty + (1+|B_\infty|)^2)\beta - 4\rho_\infty\}$.
Therefore, letting $\beta>0$ so small that
$\beta < 4\rho_\infty/(4\rho_\infty+(1+|B_\infty|)^2)$, we get
\begin{equation*}
\langle ((SL)_1+L_1) \phi,\phi \rangle
\ge c (|\phi_2|^2 + |\phi_3|^2),
\end{equation*}
where $c$ is a positive constant. This shows that
$(SL)_1+L_1\geq 0$ on $\C^{10}$ and
${\rm Ker}((SL)_1+L_1)={\rm span}\{e_1,e_8,e_9,e_{10}\}$.
Hence we have ${\rm Ker}((SL)_1+L_1)={\rm Ker}(L)$.
Thus we have verified the condition (S).

Finally, we check the condition (S*)$_1$.
We need to determine the matrix $T(\omega)=(\Pi_1Q(\omega))^T\tilde SR$
in \eqref{assump-3'}. We take $\tilde S = \beta_1 a_\infty I$. Then
the corresponding $T(\omega)$ is given by
$$
T(\omega) 
= \beta
\left(
\begin{array}{cccc}
 {0} & {0} & {0} & {0} \\
 {0} & {0} & {0} & {0} \\
 { a_\infty  \omega^T} & {0} & {0} & {0} \\
 {0} & {0} & {0} & {0} \\
\end{array}
\right).
$$
For this $T(\omega)$, we see that
$$
(SA(\omega)-T(\omega))_2
= -\frac{1}{2}\beta
\left(
\begin{array}{cccc}
 {0} & {0} & {0} & {0} \\
 {0} & {0} & {0} & {\Omega_\omega} \\
 {0} & {0} & {0} & {0} \\
 {0} & {\Omega_\omega} & {0} & {0} \\
\end{array}
\right).
$$
Therefore we obtain
$$
\langle i(SA(\omega)-T(\omega))_2 \phi, \phi \rangle
= \beta {\rm Im}(\Omega_\omega \phi_4 \cdot \bar \phi_2)
$$
for $\phi = (\phi_1,\phi_2,\phi_3,\phi_4)^T \in \mathbb{C}^{10}$.
Now let $\phi \in {\rm Ker}(L_1)$. Then $\phi = (\phi_1,0,\phi_3,\phi_4)^T$.
For this $\phi$, we have
$\langle i(SA(\omega)-T(\omega))_2 \phi, \phi \rangle = 0$, which shows
\eqref{assump-3'}.
Thus we have verified the condition (S*)$_1$.
Consequently, Theorem \ref{thm1'} is applicable to the linearized
Euler-Maxwell system and we obtain the pointwise estimate \eqref{loss-point}
and the decay estimate \eqref{loss-decay}.
This completes the proof of Theorem \ref{thmEM}.
\end{proof}

\bigskip

\noindent {\sc Acknowledgments:}\
The first author is partially supported by
Grant-in-Aid for Young Scientists (B) No.\,21740111
from Japan Society for the Promotion of Science.
The second author's research is partially supported by the Direct Grant 2010/2011 in CUHK.
The third author is partially supported by
Grant-in-Aid for Scientific Research (A) No.\,22244009.
A part of this paper was completed when Y. Ueda visited
the Institute of Mathematical Sciences, the Chinese University of Hong Kong in February, 2011.
Y. Ueda expresses sincere gratitudes to Professor Zhouping Xin for his kind invitation
and hospitality.


\end{document}